\documentclass [10pt, leqno]{amsart}
\usepackage {graphicx}
\usepackage{wrapfig}
\usepackage{amsmath,amssymb, amsthm}
\usepackage[colorlinks,linkcolor=blue,citecolor=blue,urlcolor=blue, linktocpage=true,dvips]{hyperref}

\usepackage{xcolor}

\definecolor{light-gray}{gray}{0.45}

\numberwithin{equation}{section}

\newtheorem*{theorem*}{Theorem}

\newtheorem{teorA}{Theorem}

\newtheorem{teorB}{Theorem}

\newtheorem{theor}{Theorem}
\newtheorem{theorem}{Theorem}[section]
\newtheorem{prop}[theorem]{Proposition}
\newtheorem{lemma}[theorem]{Lemma}
\newtheorem{corollary}[theorem]{Corollary}

\newtheorem{remark}[theorem]{Remark}

\def\acum{\mathcal{S}\mbox{\hskip-0.3pt}}

\def\P{\mathbf{P}}

\def\E{\mathbf{E}}

\def\uno{\text{\bf 1}}
\def\lcm{\text{\rm \,lcm}}
\begin{document}

\title[Distribution of gcd and lcm of $r$-tuples of integers]{On the probability distribution of the gcd and lcm of $r$-tuples of integers}

 \author{Jos\'{e} L. Fern\'{a}ndez}
 \address{Departamento de Matem\'{a}ticas, Universidad Aut\'{o}noma de Madrid, 28049-Madrid, Spain.}
\email{joseluis.fernandez@uam.es}

 \author{Pablo Fern\'{a}ndez}
 \address{Departamento de Matem\'{a}ticas, Universidad Aut\'{o}noma de Madrid, 28049-Madrid, Spain.}
\email{pablo.fernandez@uam.es}
\thanks{The research of both authors is partially supported by Fundaci\'{o}n Akusmatika. The second named author is partially supported by the Spanish Ministerio de Ciencia e Innovaci\'{o}n, project no. MTM2011-22851.}

\date{\today}

\begin{abstract}
{This paper is devoted to the study of statistical properties of the greatest common divisor and the least common multiple of random samples of positive integers.}
\end{abstract}

 \maketitle

\section{Introduction}\label{section:intro}

For  any given integer $n \ge 2$, let us denote by $X^{(n)}_1,
X^{(n)}_2, \ldots$ a sequence of independent random variables
uniformly distributed in $\{1, 2, \ldots, n\}$  and defined in a
certain given probability space endowed with a probability $\P$. For
a concrete realization we may take the unit interval with Lebesgue
measure and Borel $\sigma$-algebra as the probability space, and for
$j \ge 1$, the variable $ X^{(n)}_j$ whose value at $\omega\in[0,1]$
is 1 plus the $j$-th digit of the expansion in base~$n$ of $\omega$.

\smallskip
We are interested in studying the probability distributions of the random variables
$$
\gcd(X_1^{(n)},\dots, X_r^{(n)}) \quad\text{and}\quad \lcm(X_1^{(n)},\dots, X_r^{(n)}) \quad\text{for $r\ge 2$,}
$$
{\it i.e.}, the gcd and the lcm of random $r$-tuples of integers.

Dirichlet's basic and classical result asserts that the  proportion of coprime pairs of integers in $\{1, 2, \ldots, n\}$,
$$
\frac{1}{n^2} \#\big\{(i,j): 1\le i,j\le n;\ \gcd(i,j)=1\big\},
$$
tends to  ${1}/{\zeta(2)}={6}/{\pi^2}$ as $n$ tends to
$\infty$, which, in the probabilistic setting introduced above, reads:
$$
\lim_{n\to \infty}  \P\big(\gcd(X^{(n)}_1, X^{(n)}_2)=1\big)=\frac{1}{\zeta(2)}\, .
$$
See, for instance, \cite{HW}, Theorem 332.

\smallskip

The limiting behavior of the whole distribution of the $\gcd$ of
pairs follows immediately from the above: for each integer $k \ge
1$,
$$
\lim_{n \to \infty}\P\big(\gcd\big(X^{(n)}_1, X^{(n)}_2\big)=k\big)=\frac{1}{\zeta(2)}\frac{1}{k^2}\, .
$$

The probability distributions and moments of the gcd and the lcm of pairs of
integers are described asymptotically in the following two known theorems:
\begin{teorA}\label{theor:gcd of pairs}
{\rm a)} The mass function of the {\rm gcd} of pairs of integers satisfies
\begin{align}\label{eq:mass gcd pairs}
\P\big(\gcd(X^{(n)}_1,
X^{(n)}_2)=k\big)&=\frac{1}{\zeta(2)}\frac{1}{k^2}+O\Big(\frac{1+\ln(n/k)}{nk}\Big)\quad\text{for
$1\le k\le n$.}
\end{align}
{\rm b)} The moments of the {\rm gcd} of pairs of integers are given by
\begin{align}\label{eq:mean gcd pairs}
\E\big(\gcd(X^{(n)}_1,
X^{(n)}_2)\big)&=\frac{1}{\zeta(2)}\ln(n)+C+O\Big(\frac{\ln(n)}{\sqrt{n}}\Big);
\\
\E\big(\gcd(X^{(n)}_1,
X^{(n)}_2)^q\big)&=\frac{n^{q-1}}{(q+1)}\Big[\frac{2\zeta(q)}{\zeta(q+1)}-1\Big]+O(n^{q-2}
\ln(n))\,,\ \ \text{\rm for $q\ge 2$;}\!\!\label{eq:moments gcd pairs}
\end{align}

\end{teorA}
The estimate \eqref{eq:mass gcd pairs} follows directly from the usual
bounds in Dirichlet's result. Estimates \eqref{eq:mean gcd pairs} and \eqref{eq:moments gcd pairs} appear, for instance, in a paper of Cohen (take  $g(n)=1$ if $n=1$, and $g(n)=0$ elsewhere, in the notation of
Theorem in page 168 of \cite{Co}). The constant~$C$ is recorded explicitly there.
See also~\cite{ED2004}.

\begin{teorA}\label{theor:lcm of pairs}
{\rm a)} For $0< t\le 1$, the distribution function of the {\rm lcm} of pairs
of integers satisfies:
\begin{align}\label{eq:pdf lcm pairs}
\P\big(\lcm(X^{(n)}_1, X^{(n)}_2)\le t\, n^2\big)&=1-
\frac{1}{\zeta(2)}\sum_{j=1}^{\lfloor 1/t\rfloor}\frac{1-jt
(1-\ln(jt))}{j^2} +O_t\Big(\frac{\ln(n)}{n}\Big).\!\!\!\!
\end{align}
{\rm b)} The moments of the {\rm lcm} of pairs of integers satisfy
\begin{align}
\E\big(\lcm(X^{(n)}_1,
X^{(n)}_2)^q\big)&=\frac{\zeta(q+2)}{\zeta(2)(q+1)^2}\,
n^{2q}+O(n^{2q-1} \ln(n))\,,\quad\text{\rm for $q\ge
1$;}\label{eq:moments lcm pairs}
\end{align}

\end{teorA}

Observe that the bound of the error term in~\eqref{eq:pdf lcm pairs} depends on $t$. (Throughout the paper, the notation $O_t$ means that the constant in the $O$-bound depends only on $t$.)

The estimate~\eqref{eq:pdf lcm pairs} is
more involved than the corresponding result for the gcd;  it is due to Diaconis and Erd\"{o}s
\cite{ED2004}. Notice that the denominator $j^2$ in formula
\eqref{eq:pdf lcm pairs} is missing in the statement of Theorem 1 in \cite{ED2004}.

Result \eqref{eq:moments lcm pairs} can be traced back all the way to Ces\`{a}ro (see
\cite{Ce1}, page 248). See also~Theorem 2 in \cite{ED2004}.

\smallskip

In this note we will prove a number of asymptotic results (Theorem \ref{teor:pdf gcd-r} and Theorems \ref{teor:pdf lcm-r}--\ref{teor:lcm-rigual3}) concerning the probability distributions (mass distribution and moments) of the gcd and the lcm of $r$-tuples of integers, for $r\ge 3$. The case of gcd is quite direct, but not so the case of lcm, as we see later.

\smallskip
Theorem \ref{theor:gcd of pairs} can be readily extended to higher moments. It is worth recording it, as we shall use these estimates elsewhere (see \cite{FF2}).

\begin{teorB}\label{teor:pdf gcd-r}
Let $r\ge 3$.

\smallskip
\noindent {\rm a)} For $1\le k\le n$,
\begin{equation}
\P\big(\gcd(X^{(n)}_1, \dots, X^{(n)}_r)=k\big)=\frac{1}{k^r\,
\zeta(r)}+ O\Big(\frac{1}{n\, k^{r-1}}\Big).\label{eq:intro mass function gcd-r}
\end{equation}

\medskip
\noindent {\rm b)} Let $q$ be a positive integer.
\begin{itemize}
\item[{\rm b1)}]
If\/ $1\le q\le r-2$,
\begin{equation}\label{eq:intro, moments rge3, rgeq+2}
\E\big(\gcd(X_1^{(n)},\dots, X_r^{(n)})^q\big)=\frac{\zeta(r-q)}{\zeta(r)}+ O_r\Big(\frac{\ln(n)}{n}\Big).
\end{equation}
\item[{\rm b2)}] If\/ $q=r-1$,
\begin{equation}\label{eq:intro, moments rge3, reqq+1}
\E\big(\gcd\big(X^{(n)}_1, \ldots, X^{(n)}_r\big)^{r-1}\big)=\frac{\ln(n)}{\zeta(r)} +O_r(1).
\end{equation}
\item[{\rm b3)}] Finally, for $q\ge r$,
\begin{equation}\label{eq:intro, moments rge3, rleq}
\E\big(\gcd\big(X^{(n)}_1, \ldots, X^{(n)}_r\big)^{q}\big)= D_{r,q}\, n^{q-r+1} + O_{r,q}(n^{q-r} \ln(n)),
\end{equation}
where the constant $D_{r,q}$ is given by
\begin{equation*}
D_{r,q}=\frac{1}{(q+1)\,\zeta(q+1)} \sum_{k=1}^r {r\choose k} (-1)^{k+1} \, \zeta(q-r+k+1).
\end{equation*}
\end{itemize}
\end{teorB}

The estimate \eqref{eq:intro mass function gcd-r} is
straightforward, and it appears, with no bound on the error term, in Ces\`{a}ro (\cite{Ce3}, page 293). See also, for instance, \cite{Ch},
\cite{HS} and \cite{Ny}. For the sake of completeness, we prove Theorem \ref{teor:pdf gcd-r} in Section~\ref{section:gcd of r}, particularly of~\eqref{eq:intro, moments rge3, rleq}.

\smallskip
Observe that b1) implies in particular that the mean of the $\gcd$
of an $r$-tuple of integers in $\{1, 2\ldots, n\}$ has a finite
limit for $r \ge 3$:
$$
\lim_{n \to \infty} \E\big(\gcd\big(X^{(n)}_1, \dots,
X^{(n)}_r\big)\big)=\frac{\zeta(r-1)}{\zeta(r)}\, ,
$$
reflecting the fact that, on average, the $\gcd$ of an $r$-tuple is quite
close to 1, for moderately large~$r$. Notice also that the constant $D_{2,q}$ reduces to
$$
\frac{1}{(q+1)}\, \frac{1}{\zeta(q+1)}\,
\big(2\zeta(q)-\zeta(q+1)\big).
$$
as in \eqref{eq:moments gcd pairs}.
We should mention that the asymptotic estimates of moments of gcd above are valid also for non-integer $q$, but we limit ourselves to the integer~case.

\smallskip

The random behavior of the least common multiple of $r$-tuples  is subtler for $r\ge 3$ than for $r=2$. The random variable $\lcm\big(X^{(n)}_1, \dots,X^{(n)}_r\big)$ could be normalized in several manners; for instance, in terms of $n^r$, of its maximum possible value $\lcm(1,\dots,n)$, or in terms of the product of the $X_j$'s:
$$
\frac{\lcm\big(X^{(n)}_1, \dots,X^{(n)}_r\big)}{n^r},\quad\frac{\lcm\big(X^{(n)}_1, \dots,X^{(n)}_r\big)}{\lcm(1,\dots,n)},\quad\text{or}\quad \frac{\lcm\big(X^{(n)}_1, \dots,X^{(n)}_r\big)}{X^{(n)}_1 \cdots X_r^{(n)}}.
$$
In all three alternatives, we obtain a random variable with values in $[0,1]$.
We will focus on the first alternative (but see Propositions \ref{prop:prob_lcm_divide_by_product_r} and \ref{prop:moments_lcm_divide_by_product_r} for the third one).
Recall that Theorem~\ref{theor:lcm of pairs} claims that the sequence of variables $$
\mathcal{L}_n=\dfrac{\lcm(X^{(n)}_1, X^{(n)}_2)}{n^2}\, ,\quad n \ge 1$$ converges in distribution to a random variable $\mathcal{L}$ with values in $[0,1]$ whose complementary distribution function is given by
$$
\P(\mathcal{L} >t)=\frac{1}{\zeta(2)}\sum_{j=1}^{\lfloor 1/t\rfloor}\frac{1-jt
(1-\ln(jt))}{j^2}\, ,
$$
and whose moments are given by $\E(\mathcal{L}^q)=\frac{\zeta(q+2)}{\zeta(2) (q+1)^2}, \ q \ge 1$.

\smallskip

To state our results about the distribution function of the lcm, we introduce the following notation: for $r\ge 2$ and $s> 0$, denote by
\begin{equation}\label{eq:def of Ar}
\mathcal{A}_r(s):=\big\{(x_1,\dots, x_r): 0\le x_1,\dots, x_r\le
1\,, \ x_1\cdots x_r\le s\big\}
\end{equation}
the part of the unit positive $r$-cube where $x_1\cdots
x_r\le  s$. We write $\Omega_r(s)$ for the volume of $\mathcal{A}_r(s)$. Observe that $\Omega_r(s)=1$ for $s\ge 1$. For $s<1$,
\begin{equation}
\label{eq:formula for Omegar}
\Omega_r(s)=s\sum_{k=0}^{r-1} \frac{\ln(1/s)^k}{k!}
\end{equation}
(see Lemma \ref{lemma:volumen bajo hiperbola}).
For each $r\ge 2$, write
\begin{equation}
\label{eq:def of Tr}
T_r:=\prod_p \Big(1-\frac{1}{p}\Big)^{r-1}\, \Big(1+\frac{r-1}{p}\Big).
\end{equation}

The constant $T_r$ is the asymptotic proportion of $r$-tuples of integers that are pairwise coprime (see~\cite{CB2001} and~\cite{To2004}). Clearly, $\lim_{r\to \infty} T_r=0$; in fact, it  does so very rapidly: $\lim_{r \to \infty}T_r^{1/r} \ln(r)=e^{-\gamma}$, where $\gamma$ is Euler's constant, see \cite{Hwang}. The first values are: $T_2=1/\zeta(2)\approx 0.60793$, $T_3\approx 0.28675$, $T_4\approx 0.11488$, etc.
\begin{theor}\label{teor:pdf lcm-r} Let $r\ge 3$. Then,
for $0< t\le 1$,
\begin{align}\label{eq:intro pdf lcm rge2-conzeta}
\liminf_{n\to\infty}\P\big(\lcm\big(X^{(n)}_1, &\dots,X^{(n)}_r\big)\le t n^r\big)
\ge 1-\frac{1}{\zeta(r)}
\sum_{j=1}^\infty
\frac{1-\Omega_r(t\, j^{r-1})}{j^r}\,,
\\
\label{eq:intro pdf lcm rge2-conTr}
\limsup_{n\to\infty}\P\big(\lcm\big(X^{(n)}_1, &\dots,X^{(n)}_r\big)\le t n^r\big)
\le 1-T_r\,
\sum_{j=1}^\infty
\frac{1-\Omega_r(t\, j^{r-1})}{j^r}\,.
\end{align}

\end{theor}

As $\Omega_r(s)=1$ for $s\ge 1$, the series in \eqref{eq:intro pdf lcm rge2-conzeta} and \eqref{eq:intro pdf lcm rge2-conTr}, for each $0\le t\le 1$, are actually finite sums (the range extends to those $j$ such that $j^{r-1}\le1/t$). Notice also that the right hand side of~\eqref{eq:intro pdf lcm rge2-conzeta} is a distribution function, with value 0 as~$t\to 0$, and value 1 as $t\to 1$. The right hand side of \eqref{eq:intro pdf lcm rge2-conTr} is not a distribution function, as it takes the value $1-T_r\zeta(r)$ as $t\to 0$ (see Figure \ref{comparison case r=3} for a depiction of the case~$r=3$).

\smallskip

Setting $r=2$ in Theorem \ref{teor:pdf lcm-r}, we  recover \eqref{eq:pdf lcm pairs}, since $T_2={1}/{\zeta(2)}$ and $ \Omega_2(s)=s(1-\ln(s))$.

\medskip

For the moments of the lcm of $r$-tuples, we  prove:
\begin{theor}\label{teor:moments lcm-r}
Let $r\ge 3$. For each integer $q\ge 1$,
\begin{align}\label{eq:intro, moments lcm rge2-conzeta}
\limsup_{n\to\infty}
\frac{\E\big(\lcm\big(X^{(n)}_1, \dots,X^{(n)}_r\big)^q\big)}{n^{rq}}&\le
\frac{1}{\zeta(r)}\ \frac{\zeta(r(q+1)-q)}{(q+1)^r} \,. 
\\
\label{eq:intro, moments lcm rge2-conTr}
\liminf_{n\to\infty}
\frac{\E\big(\lcm\big(X^{(n)}_1, \dots,X^{(n)}_r\big)^q\big)}{n^{rq}}&\ge
T_r\ \frac{\zeta(r(q+1)-q)}{(q+1)^r} \,. 
\end{align}
\end{theor}
\noindent Again, for $r=2$, since $T_2=1/\zeta(2)$, we recover \eqref{eq:moments lcm pairs} of Theorem \ref{theor:lcm of pairs}.

\medskip
Theorems \ref{teor:pdf lcm-r} and \ref{teor:moments lcm-r} do tell us that
$$
\P\big(\lcm\big(X^{(n)}_1, \dots,X^{(n)}_r\big)> t n^r\big)\asymp
\sum_{j=1}^\infty
\frac{1-\Omega_r(t\, j^{r-1})}{j^r}\,,
$$
and
$$
\E\big(\lcm\big(X^{(n)}_1, \dots,X^{(n)}_r\big)^q\big)\asymp n^{rq}\,,
$$
but, asymptotically, they only provide  upper and lower estimates for the distribution function and moments of the lcm of $r$-tuples of integers in the limit $n\to\infty$. We have not been able to handle the combinatorics that would lead to establish precise asymptotic estimates for general $r$; but we have studied in detail the case~$r=3$ and obtained the precise result contained in Theorem~\ref{teor:lcm-rigual3}. To state it, we need to introduce some more notation. We will denote by $\omega(m)$ the number of \textit{distinct} prime factors of $m$, and for each $r\ge 2$ we will write $\Upsilon_r(m)$ for the arithmetic function given by~$\Upsilon_r(1)=1$ and
\begin{equation}\label{eq:denition of Delta}
\Upsilon_r(m)=\prod_{p|m} \frac{(1+(r-2)/p)}{(1+(r-1)/p)}\quad\text{for $m\ge 2$.}
\end{equation}
The function $\Upsilon_r$ is multiplicative, and $\Upsilon_r(m)<1$, for $m>1$. The case  $r=3$,
$$
\Upsilon_3(m)=\prod_{p|m}\frac{1+1/p}{1+2/p},
$$
will be of special interest. Finally, we shall denote by $J$ the Dirichlet series:
\begin{equation}\label{eq:definition of T(s)}
J(s)=\sum_{m=1}^\infty \frac{\Upsilon_3(m)\, 3^{\omega(m)}}{m^{s}}\,,   \quad \Re(s)>1\,.
\end{equation}

\begin{theor}\label{teor:lcm-rigual3}
{\rm a)} For $0\le t\le 1$,
\begin{align}\nonumber
\lim_{n\to\infty}
\P\big(\lcm\big(X^{(n)}_1, X^{(n)}_2, &\,X^{(n)}_3\big)\le t n^3\big)
\\ \label{eq:pdf lcm-rigual3}
&=1-T_3\sum_{j=1}^\infty \frac{1}{j^3} \, \sum_{m=1}^{\infty} \frac{\Upsilon_3(m)\, 3^{\omega(m)}}{m^2}\ \big(1-\Omega_3(tj^2m)\big),
\end{align}

\smallskip
{\rm b)} For integer $q\ge 1$,
\begin{equation}\label{eq:moments lcm-rigual3}
\lim_{n\to\infty}\frac{\E\big(\lcm(X^{(n)}_1,
X^{(n)}_2, X^{(n)}_3)^q\big)}{n^{3q}}=T_3\,\frac{\zeta(2q+3)}{(q+1)^3} \ J(q+2)\,,
\end{equation}
\end{theor}
Incidentally, observe that the case $t=0$ of \eqref{eq:pdf lcm-rigual3} implies the identity
\begin{equation}
 J(2)=\sum_{m=1}^{\infty} \frac{\Upsilon_3(m)\, 3^{\omega(m)}}{m^2}=\frac{1}{T_3\, \zeta(3)}
\end{equation}
(see Remark \ref{remark:check T(2)}).
Figure \ref{comparison case r=3} compares the exact value \eqref{eq:pdf lcm-rigual3} with the lower and upper bounds given in \eqref{eq:intro pdf lcm rge2-conzeta} and \eqref{eq:intro pdf lcm rge2-conTr} for $r=3$.
\begin{figure}
\begin{center}
\resizebox{6.5cm}{!}{\includegraphics{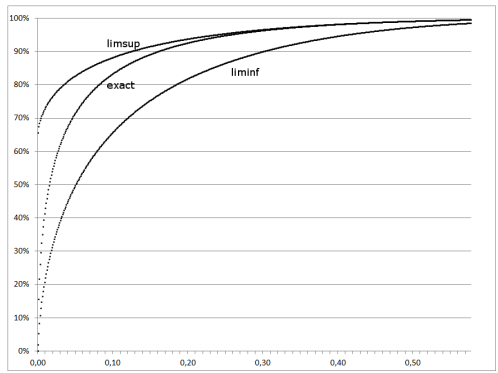}}
\end{center}
\caption{Comparison for the case $r=3$.}
\label{comparison case r=3}
\end{figure}

Finally, we  discuss certain \textit{waiting time} questions concerning sequences of successive gcds and lcms 
\begin{align*}
&z_1=x_1,\ z_2=\gcd(x_1,x_2),\ z_3=\gcd(x_1,x_2,x_3)\, \dots,
\\
&w_1=x_1,\ w_2=\lcm(x_1,x_2),\ w_3=\lcm(x_1,x_2,x_3)\, \dots,
\end{align*}
of integers $x_1,x_2,\dots$ drawn uniformly and independently from $\{1,\dots,n\}$.

The sequence $(z_j)$ decreases almost surely to 1, while the sequence $(w_j)$  increases almost surely to the number $\lcm(1,\dots,n)$. Their respective \textit{expected} waiting times are dealt with in Section~\ref{sec:waiting times}.

\smallskip
The paper is organized as follows:
we will analyze the properties of the gcd of $r$-tuples in Section~\ref{section:gcd of r}.
Section~\ref{section:lcm of r} contains the proofs of the results concerning the lcm of $r$-tuples. The particular case $r=3$ for lcm is studied in Section~\ref{lcm for r=3}. 
Finally, Section~\ref{sec:waiting times} contains the analysis of those waiting times related to gcd and~lcm.

\medskip
\textbf{Acknowledgments}. We thank Fernando Chamizo from the Universidad Aut\'{o}\-no\-ma de Madrid for some fruitful discussions.

\section{Preliminaries and notation}\label{sec:preliminaries}

Throughout the paper, for a vector $\mathbf{x}=(x_1, \ldots, x_r)$ of positive integers, we abbreviate $\gcd(\mathbf{x})=\gcd(x_1, \ldots, x_r)$ and $\lcm(\mathbf{x})=\lcm(x_1, \ldots, x_r)$. Also if $\boldsymbol{\beta}=(\beta_1, \ldots, \beta_r)$ is a vector of positive real numbers we economize and write $\mathbf{x} \le \boldsymbol{\beta}$ to mean $1 \le x_j \le \beta_j$, for $j=1, \ldots r$. Even further, for a positive number $z$ we simplify and write $\mathbf{x}\le z$ to mean that $1\le x_j \le z$, for $j=1, \ldots, r$.

\smallskip

We shall encounter a few times in what follows sums of the form
$$
\sum_{\mathbf{x}\le n} f\big(\gcd(\mathbf{x})\big),
$$
where $r\ge 1$ and $f$ is some arithmetic function. They can be readily seen to be
\begin{equation}\label{eq:Cesaro sum}
\sum_{\mathbf{x}\le n} f\big(\gcd(\mathbf{x})\big)=\sum_{j=1}^n (\mu* f)(j)\, \Big\lfloor \frac{n}{j}\Big\rfloor^r.
\end{equation}
Here, $\mu$ denotes the M\"{o}bius function and the symbol $*$ stands
for the {Dirichlet convolution}. The expression above is valid also for $r=1$, with the conventional understanding that $\gcd(j)=j$, for any integer $j\ge 1$.

Equation \eqref{eq:Cesaro sum} is sometimes referred as \textit{Ces\`{a}ro's formula} (see \cite{Ce1}, \cite{Ce3}). Its simple proof follows:

\begin{proof}[Proof of \eqref{eq:Cesaro sum}] Using the properties of the M\"{o}bius function, we have that, for any arithmetical function $F$,
\begin{equation}
\label{eq:use of mu}
\sum_{\mathbf{x}\le n, \,\gcd(\mathbf{x})=1} F(\mathbf{x})=\sum_{k=1}^n \mu(k) \sum_{\mathbf{x}\le n, \,k|\mathbf{x}} F(\mathbf{x})
=\sum_{k=1}^n \mu(k) \sum_{\mathbf{y}\le n/k} F(k\mathbf{y}).
\end{equation}
In our case,
\begin{align*}
\sum_{\mathbf{x}\le n} f\big(\gcd(\mathbf{x})\big)&=\sum_{d=1}^n f(d)\, \sum_{\mathbf{x}\le n, \gcd(\mathbf{x})=d} 1
=\sum_{d=1}^n f(d)\, \sum_{\mathbf{y}\le n/d, \gcd(\mathbf{y})=1} 1
\\
&=\sum_{kd\le n} f(d) \mu(k)\Big\lfloor\frac{n/d}{k}\Big\rfloor^r
=\sum_{j=1}^n \Big\lfloor\frac{n}{j}\Big\rfloor^r \sum_{kd=j} f(d)\, \mu(k)\,.\qedhere
\end{align*}

\end{proof}

We will use the following notation for the \textit{summatory function} $\acum{\alpha}$ of an arithmetic function $\alpha$:
$$
\acum{\alpha}(x)=\sum_{j\le x} \alpha(j)\quad\text{for each $x>0$.}
$$

The following arithmetic function:
\begin{equation}\label{eq:def de hr}
D_r(m)=m^r -(m-1)^r\,,\quad m\ge 1
\end{equation}
will appear several times in the paper; its summatory function is given by
\begin{equation}\label{eq:acumulador de hr}
\acum{D}_r(m)=m^r.
\end{equation}
Observe that $D_r(m)\le r m^{r-1}$.

For each integer $q\ge 1$, we denote  $I_q(n)=n^q$, $n \ge 1$; its summatory function satisfies
\begin{equation}\label{eq:acumulador de fq}
\acum{I_q}(m)=\sum_{k=1}^m I_q(k)=\sum_{k=1}^m k^q =\frac{m^{q+1}}{q+1} +O_q(m^q)\,.
\end{equation}

For each integer $q\ge 1$, denote by $\varphi_q$ the arithmetic function given by $\varphi_q(n)=(\mu* I_q)(n)$ (the so-called $q$-\textit{Jordan totient function}).
\begin{lemma}\label{lemma:acumulador de phiq}
For integer $q\ge 1$, the summatory function of $\varphi_q$ satisfies
\begin{equation}
\acum{\varphi_q}(n)=\begin{cases}
\dfrac{1}{(q+1)\,\zeta(q+1)}\, n^{q+1}+O_q(n^q) \quad \text{if $q \ge 2$,}
\\[9pt]
\dfrac{1}{2 \zeta(2)}\,n^2+O(n \ln n)  \quad \text{if $q=1$.}
\end{cases}
\end{equation}
\end{lemma}
\noindent Notice that  $\varphi_1$ is just Euler's $\varphi$ function and that the case $q=1$ of Lemma \ref{lemma:acumulador de phiq} is just Dirichlet's theorem.
\begin{proof}
Apply the well-known expression (see \cite{Ap}, Theorem 3.10) for the summatory function of the Dirichlet convolution of two arithmetical functions $\alpha$ and $\beta$,
\begin{equation}
\label{eq:convolution accumulator}
\acum\,({\alpha*\beta})(x)=\sum_{j\le x} \alpha(j) \ \acum{\beta}\big(x/j\big)=\sum_{j\le x} \beta(j) \ \acum{\alpha}\big(x/j\big).
\end{equation}
and equation \eqref{eq:acumulador de fq}. 
\end{proof}

The following is an elementary calculus lemma which shall be useful in the discussion of distributional properties of lcm.
\begin{lemma}\label{lemma:volumen bajo hiperbola}
For each $r\ge 1$ and for all $s> 0$, denote by $\Omega_r(s)$ the volume of the
$r$-dimensional set $
\mathcal{A}_r(s):=\big\{(x_1,\dots, x_r): 0\le x_1,\dots, x_r\le
1\,, \ x_1\cdots x_r\le s\}$. Then $\Omega_r(s)=1$ for $s\ge 1$ and
\begin{equation*}
\Omega_r(s)=s\sum_{j=0}^{r-1} \frac{\ln(1/s)^j}{j!}\quad\text{for $s<1$.}
\end{equation*}
\end{lemma}
\begin{proof}
Observe that $\Omega_1(s)=s$ and that
$
\Omega_r(s)=s+\int_s^1 \Omega_{r-1}(s/x)\, dx
$ for $r\ge 2$.\end{proof}

A simple change of variables gives that if $0 \le \beta_1, \ldots, \beta_r\le 1$, then
\begin{equation}\label{eq:vol with betas}
\textrm{Vol}\big\{(x_1, \ldots, x_r): \ 0 \le x_j \le \beta_j, j=1, \ldots, r, \ x_1\cdots x_r \le s\big\}=B \,\Omega_r\big({s}/{B}\big)\, ,
\end{equation}
where $B=\prod_{j=1}^r \beta_j$. Observe that if $B\le s$, the statement is obvious.
\medskip

Later on, Section \ref{sec:waiting times}, while discussing waiting times, we shall need the standard Euler's extension to real argument of the harmonic numbers given by $H_n=\sum_{j=1}^n{1}/{j}$ for integer $n \ge 1$; and the standard approximation of $H_n$ by $\ln(n)$, which we record in the following:
\begin{lemma}\label{lemma:harmonic_numbers}
The parametric integral defined  for any real $a>0$ by
$$
H(a)=\int_0^{\infty} \big[1-(1-e^{-t})^a\big] dt
$$
satisfies
$$
H(a)=\ln(a)+\gamma+O\Big(\frac{1}{a}\Big)\, , \quad \text{as} \quad a \to \infty\, .
$$
Besides, of course, for any integer $n \ge 1$, $H(n)=H_n$.\end{lemma}

\section{Probability distribution of the gcd of $r$-tuples}\label{section:gcd of r}
In the present  section, we prove Theorem~\ref{teor:pdf gcd-r}, which describes the mass function and the moments of the random variable
$\gcd(X_1^{(n)},\dots, X_r^{(n)}),$ for $r\ge 3$.

\smallskip
Ces\`aro's formula \eqref{eq:Cesaro sum}, with $f=\delta_1$, where $\delta_1(n)=1$ if $n=1$ and is $0$ elsewhere, yields
\begin{equation}\label{eq:prob gcd=1 in terms of Cesaro}
\P(\gcd(X_1^{(n)},\dots, X_r^{(n)})=1) 
=\frac{1}{n^r}\sum _{j=1}^n \mu(j) \Big\lfloor\frac{n}{j}\Big\rfloor^r\, .
\end{equation}
which  readily gives that there exists a constant $C_r>0$ such that
$$
\Big|\P\big(\gcd(X^{(n)}_1, \dots, X^{(n)}_r)=1\big)-\frac{1}{\zeta(r)}\Big| \le C_r\,\frac{1}{n}\quad\text{for any integer $n \ge 1$.}
$$
As $\gcd(x_1,\dots,x_r)=k$ means that $k|x_1,\dots, k|x_r$ and
$\gcd(x_1/k,\dots,x_r/k)=1$, we deduce that there exists a constant $\widetilde C_r>0$ such that
\begin{equation*}
\Big|\P\big(\gcd\big(X^{(n)}_1,
\dots,X^{(n)}_r\big)=k\big)-\frac{1}{k^r\, \zeta(r)}\Big| \le \widetilde C_r\
\frac{1}{n\, k^{r-1}},
\end{equation*}
which is statement a) of Theorem \ref{teor:pdf gcd-r}.

\smallskip

Next, we turn to moments. 
Statements b1) and b2) are immediate consequences of
the estimate~\eqref{eq:intro mass function gcd-r}: just write
\begin{align*}
\E\big(\gcd(X_1^{(n)},\dots, X_r^{(n)})^q\big)=\frac{1}{\zeta(r)}\sum_{k=1}^n \frac{k^q}{k^r} +O_r\Big(\frac{1}{n}\sum_{k=1}^n \frac{k^q}{k^{r-1}}\Big).
\end{align*}
For $q\le r-2$, the sum above tends to $\zeta(r-q)/\zeta(r)$ as $n\to\infty$ with an error bound $O_r(1/n^{r-q-1})$, which in the worst case ($r=q+2$) is $O_r\big(1/n\big)$; while the right hand side $O$ term is, again in the worst case,   $O_r(\ln(n)/n).$

If $q=r-1$, we get that $\E\big(\gcd(X_1^{(n)},\dots,
X_r^{(n)})^{r-1}\big)=\ln(n)/{\zeta(r)} +O_r(1)$.

\medskip

b3) The argument above would not work for $q\ge r$, since the purported error term turns out to be of the same order as the main term.
Using Ces\`{a}ro's formula~\eqref{eq:Cesaro sum}, the identities \eqref{eq:acumulador de hr} and~\eqref{eq:convolution accumulator}, and Lemma~\ref{lemma:acumulador de phiq}, 
one writes
\begin{align*}
\E\big(\gcd\big(X^{(n)}_1, \ldots, &\,X^{(n)}_r\big)^{q}\big)
=\frac{1}{n^r} \sum_{j=1}^n \varphi_q(j)\, \Big\lfloor \frac{n}{j}\Big\rfloor^r
= \frac{1}{n^r}\, \sum_{j=1}^n \, D_r(j) \,\acum{\varphi_q}(\lfloor n/j\rfloor )
\\
&=\frac{1}{n^r} \,\frac{1}{(q+1)\, \zeta(q+1)}\sum_{j=1}^n D_r(j) \Big\lfloor\frac{n}{j}\Big\rfloor^{q+1}+O_q\Big(\frac{1}{n^r} \sum_{j=1}^n D_r(j) \frac{n^q}{j^q}\Big)
\\
&= \frac{n^{q-r+1}}{(q+1)\, \zeta(q+1)}\sum_{j=1}^n \frac{D_r(j)}{j^{q+1}} +O_q\Big(n^{q-r} \sum_{j=1}^n \frac{D_r(j)}{j^q}\Big)\, .
\end{align*}
As $D_r(j)\le r \, j^{r-1}$ and $q\ge r\ge 3$, this last error term is at most $O_{q,r}(n^{q-r} \ln(n))$. 
Finally, using the definition \eqref{eq:def de hr} of $D_r(j)$, and recalling again that $q\ge r$, we get that
\begin{align*}
&\sum_{j=1}^n \frac{D_r(j)}{j^{q+1}}=\sum_{j=1}^n \frac{j^r-(j-1)^r}{j^{q+1}}=\sum_{j=1}^n\frac{1}{j^{q+1}} \sum_{k=1}^r {r\choose k} (-1)^{k+1} j^{r-k}
\\
&\ \ =\sum_{k=1}^r {r\choose k} (-1)^{k+1} \sum_{j=1}^n \frac{1}{j^{q+1-r+k}}=\sum_{k=1}^r {r\choose k} (-1)^{k+1} \zeta({q-r+k+1})+O_{q,r}\Big(\frac{1}{n}\Big)\,,
\end{align*}
and this proves \eqref{eq:intro, moments rge3, rleq}.

\section{Probability distribution of the lcm of $r$-tuples}\label{section:lcm of r}
Observe that
\begin{equation}\label{eq:lcm as a product}
\lcm(a_1,\dots, a_r)=\frac{a_1\cdots a_r}{\prod_{i< j} \gcd(a_i,a_j)}\, \frac{\prod_{i< j< k} \gcd(a_i,a_j, a_k)}{\prod_{i< j< k< l} \gcd(a_i,a_j, a_k, a_l)}\cdots
\end{equation}
Notice that the lcm is the product of the numbers,
$\lcm(a_1,\dots, a_r)=a_1\cdots a_r$, if and only if they are \textit{pairwise coprime}, that is, $\gcd(a_i,a_j)=1$ for each $i\ne j$. For $r=2$, of course, there is no difference between coprimality and pairwise coprimality.

\subsection{Pairwise coprimality and equidistribution}
For a $r$-tuple of positive integers $\mathbf{x}$, we write $\mathbf{x}\in {\rm PC}$ if the $r$ components of $\mathbf{x}$ are pairwise coprime (that is, $\gcd(x_i,x_j)=1$ for $i\ne j$).
The following result was obtained by Toth and also by Cai--Bach (see \cite{To2004} and~\cite{CB2001}):
\begin{lemma}
\label{lemma:PC}For each $r\ge 2$,
\begin{equation}\label{eq:definition of Tr}
\lim_{n\to\infty} \P\big(\{X_1^{(n)}, \dots, X_r^{(n)}\}\in {\rm PC}\big)=\prod_{p} \Big(1-\frac{1}{p}\Big)^{r-1}\Big(1+\frac{r-1}{p}\Big):=T_r\,.
\end{equation}
\end{lemma}
When $r=2$, the constant  $T_2$ is $1/\zeta(2)$ and we recover \eqref{eq:intro mass function gcd-r}.

Extending the argument of Cai and Bach, one could prove that the PC $r$-tuples are, in fact, equidistributed, in the following sense (see the details in \cite{FF1}):
\begin{lemma}
\label{lemma:PC equidistributed} Fix $r\ge 2$. Then, for any function $f\in C([0,1]^r)$,
$$
\lim_{n\to\infty} \ \frac{1}{n^r} \sum_{{\mathbf{x}\le n,\, \mathbf{x}\in \text{\rm PC}}} f\Big(\frac{x_1}{n}, \dots, \frac{x_r}{n}\Big)= T_r\, \int_{[0,1]^r} f(u_1,\dots, u_r)\, du_1\cdots du_r.
$$
\end{lemma}
Taking $f\equiv 1$, we recover Lemma \ref{lemma:PC}. Let us record now two applications of Lemma~\ref{lemma:PC equidistributed} which will be useful in forthcoming arguments. Fix $0< t\le 1$  and $\boldsymbol{\beta}\in\mathbb{R}^r$ with $0 \le \beta_1, \ldots, \beta_r \le 1$. Let $B=\prod_{j=1}^r \beta_j$ and define $\mathcal{B}=[0,\beta_1]\times\cdots \times [0,\beta_r]$. Choose $f=\uno_{\mathcal{B}\,\cap\, \mathcal{A}_r(t)}$, where $\mathcal{A}_r(t)$ the region given in \eqref{eq:def of Ar}.
Then, thanks to~\eqref{eq:vol with betas},
\begin{equation}
\label{eq:equidistribution of PC in Ar}
\lim_{n\to\infty}\frac{1}{n^r}\, \#\{\mathbf{x} \le n\,\boldsymbol{\beta};\mathbf{x}\in \textrm{PC}; x_1\cdots x_r\le t n^r\} = T_r\, B\, \Omega_r({t}/{B})\,,
\end{equation}
where $\Omega_r$ is the function given in \eqref{eq:formula for Omegar}.

\smallskip
For $q\ge 1$, take $f(u_1,\dots, u_r)=u_1^q\cdots u_r^q \cdot \uno_\mathcal{B}$. Then we get
\begin{equation}
\label{eq:equidistribution for moments}
\lim_{n\to\infty}\frac{1}{n^{r+qr}}\, \sum_{\substack{\mathbf{x}\le n \, \boldsymbol{\beta}\\ \mathbf{x}\in \textrm{PC}}} x_1^q\cdots x_r^q = T_r\, \frac{1}{(q+1)^r}\, B^{q+1}\,.
\end{equation}
Notice that actually the functions considered to derive these two examples are not continuous, as demanded by Lemma~\ref{lemma:PC equidistributed}, but a standard approximation argument yields the results: for instance, for \eqref{eq:equidistribution of PC in Ar}, consider
$$
f_\varepsilon=\Big(1-\dfrac{\text{dist}\big(\bullet\,,\mathcal{B}\cap\mathcal{A}_r(t)\big)}{\varepsilon}\Big)^+\, ,
$$
to let $\varepsilon \to 0$.

\smallskip

The following extension of Lemma~\ref{lemma:PC equidistributed} will be useful in Section~\ref{lcm for r=3}. Fix an $r$-tuple $\mathbf{a}=(a_1,\dots, a_r)$ $\in\textrm{PC}$. We will say that $\mathbf{x}\in \textrm{PC}_\mathbf{a}$ if the components $x_j$ are pairwise coprime and, additionally, $\gcd(x_1,a_1)=\cdots=\gcd(x_r,a_r)=1$. We refer,  again, the reader to~\cite{FF1}.

\begin{lemma}
\label{lemma:PCE equidistributed} Fix $r\ge 2$ and consider an $r$-tuple $\mathbf{a}=(a_1,\dots, a_r)\in\text{\rm PC}$. Then, for any function $f\in C([0,1]^r)$,
$$
\lim_{n\to\infty} \ \frac{1}{n^r} \sum_{{\mathbf{x}\le n,\, \mathbf{x}\in \text{\rm PC}_\mathbf{a}}} f\Big(\frac{x_1}{n}, \dots, \frac{x_r}{n}\Big)= T_r\, \, \Upsilon_r\Big(\prod_{j=1}^r a_j\Big)\int_{[0,1]^r} f(u_1,\dots, u_r)\, du_1\cdots du_r,
$$
where the function $\Upsilon_r$ is given by \eqref{eq:denition of Delta}.
\end{lemma}

Taking $f\equiv 1$, we obtain that the proportion of $\textrm{PC}_\mathbf{a}$ $r$-tuples in $\{1,\dots,n\}^r$ tends to $T_r\,\Upsilon_r(\prod_{j=1}^r a_j)$ as $n\to\infty$.
Again, two special cases of interest:
\begin{align}
\label{eq:equidistribution of PCE in Ar}
\lim_{n\to\infty}\frac{1}{n^r}\, \#\{\mathbf{x}\le n \, \boldsymbol{\beta}; & \ \mathbf{x} \in \text{\rm PC}_\mathbf{a}; \ x_1\cdots x_r\le t n^r\}
\\\nonumber
&= T_r\, \Upsilon_r\Big(\prod_{j=1}^r a_j\Big)\, B\, \Omega_r({t}/{B})\, ,\quad\text{for $t\le 1$ fixed,}
\end{align}
and for $q\ge 1$,
\begin{equation}
\label{eq:equidistribution for moments PCE}
\lim_{n\to\infty}\frac{1}{n^{r+qr}}\!\!\! \sum_{{\mathbf{x}\le n \, \boldsymbol{\beta}\,, \mathbf{x}\in \text{\rm PC}_\mathbf{a}}} \!\!\!x_1^q\cdots x_r^q = T_r\, \Upsilon_r\Big(\prod_{j=1}^r a_j\Big)\, \frac{1}{(q+1)^r}\, B^{q+1}\,.
\end{equation}

\subsection{Distribution function of the lcm of $r$-tuples}
In this subsection we shall prove Theorem \ref{teor:pdf lcm-r}. Fix $r\ge 3$.

\subsubsection{Proof of \eqref{eq:intro pdf lcm rge2-conzeta}}   For $0 \le t \le 1$ and real $z\ge 1$, we introduce  the following counting functions:
\begin{align}
N_r(t,z)&=\#\{\mathbf{x}\le z :\  x_1\cdots x_r\le t z^r\}\label{eq:def of N}
\\
G_r(t,z)&=\#\{\mathbf{x}\le z :\ x_1\cdots x_r\le t z^r\,,\ \gcd(\mathbf{x})=1\}\label{eq:def of M}
\\
L_r(t,z)&=\#\{\mathbf{x}\le z :\, \lcm(\mathbf{x})\le t z^r\}\label{eq:def of K}
\end{align}
Our objective is to estimate
$$
\P\big(\lcm\big(X^{(n)}_1, \dots,X^{(n)}_r\big)\le t n^r\big)= \frac{L_r(t,n)}{n^r}\quad\text{as $n\to\infty$}\, ,
$$
and we shall use for that purpose convenient auxiliary estimates for $N$ and~$G$.

The following elementary lemma estimates the number of lattice
points in the positive $r$-cube $[1,z]^r$ such that $x_1\cdots x_r\le  t z^r$,  in
terms of the volume of the region: 
\begin{lemma}
\label{lemma:N in terms of volume} For $0<\delta\le t< 1$,
$$
\frac{1}{z^r} \, N_r(t,z)=\Omega_r(t)+O_\delta \Big(\frac{1}{z}\Big)\quad\text{as $z\to\infty$.}
$$
\end{lemma}
Observe that the bound dependes on $\delta$ and not on $t$.
The next lemma estimates the function $G_r(t,z)$:
\begin{lemma}
\label{lemma:estimate of M} {\rm a)} For $0<\delta\le t< 1$, we have that, as $z\to\infty$,
\begin{equation}
\label{eq:estimate for M, t<1}G_r(t,z)=\frac{\Omega_r(t)}{\zeta(r)}\, z^r + O_\delta(z^{r-1}).
\end{equation}
{\rm b)} For $t\ge 1$,
\begin{equation}
\label{eq:estimate for M, t>1}G_r(t,z)=\frac{1}{\zeta(r)}\, z^r + O(z^{r-1}).
\end{equation}
\end{lemma}
\begin{proof}
The estimate \eqref{eq:estimate for M, t>1} is the case $k=1$ of \eqref{eq:intro mass function gcd-r}, as for $t\ge 1$, $G_r(t,z)$ counts coprime $r$-tuples.
For the case $t<1$, partitioning accordingly as the
value of the greatest common divisor, we can write
\begin{align*}
N_r(t,z)&=\sum_{d\le z} \#\{\mathbf{x}\le z :\ x_1\cdots x_r\le t z^r\,,\ \gcd(\mathbf{x})=d\}
\\
&=\sum_{d\le z} \#\{\mathbf{y}\le z/d :\ y_1\cdots y_r\le t (z/d)^r\,,\ \gcd(\mathbf{y})=1\}
=\sum_{d\le z}G_r(t, z/d).
\end{align*}
By M\"{o}bius inversion (see, for instance, Theorem~2.22 in~\cite{Ap}), 
we deduce that
$$
G_r(t,z)=\sum_{d\le z} \mu(d)\, N_r(t,z/d).
$$
Now, for $t\ge \delta$, using Lemma \ref{lemma:N in terms of volume} and $r\ge 3$, we get that,
\begin{align*}
G_r(t,z)&= \sum_{d\le z} \mu(d)\, \Big(\frac{z^r}{d^r}\, \Omega_r(t)+O_\delta \Big(\frac{z^{r-1}}{d^{r-1}}\Big)\Big)=\Omega_r(t) z^r \sum_{d\le z} \frac{\mu(d)}{d^r} +O_\delta\big(z^{r-1}\big)
\\
&=\Omega_r(t) \frac{z^r}{\zeta(r)}  +O_\delta\big(z^{r-1}\big),
\end{align*}as claimed.
\end{proof}

Our objective is to estimate~$L_r(t,z)$.
Let us write, again by partitioning,
\begin{align*}
&L_r(t,z)=\sum_{d\le z} \#\{\mathbf{x} \le z:\  \lcm(\mathbf{x})\le t z^r,\, \gcd(\mathbf{x})=d\}
\\
&\quad=\sum_{d\le z} \#\{\mathbf{y} \le z/d:\  \lcm(\mathbf{y})\le td^{r-1} \Big(\frac{z}{d}\Big)^r,\, \gcd(\mathbf{y})=1\}
\\
&\quad \ge \sum_{d\le z} \#\{\mathbf{y} \le z/d:\  y_1 \cdots y_r\le td^{r-1} \Big(\frac{z}{d}\Big)^r,\, \gcd(\mathbf{y})=1\}=\sum_{d\le z} G_r(t\, d^{r-1}, z/d)\, ,
\end{align*}
where we have used that if $\gcd(\mathbf{x})=d$ and if we write $x_i=dy_i$, for $i=1,\dots, r$, then $\gcd(\mathbf{y})=1$ and $\lcm(\mathbf{y})=\lcm(\mathbf{x})/d$. We have also used that $y_1 \cdots y_r \ge \lcm(\mathbf{y})$. The above inequality would turn into an equality if $r=2$.

Therefore,
$$
\P\big(\lcm\big(X^{(n)}_1, \dots,X^{(n)}_r\big)\le t n^r\big)= \frac{L_r(t,n)}{n^r}\ge \frac{1}{n^r}\sum_{d\le n} G_r(t\, d^{r-1}, n/d).
$$

To get the lower bound \eqref{eq:intro pdf lcm rge2-conzeta}, we estimate the latter sum. For that purpose, we split it into two sums, depending on whether the first argument of $G_r$ is less than or equal to 1, or not:
$$
\sum_{d\le n} G_r\big(t\, d^{r-1}, {\textstyle\frac{n}{d}}\big)=\sum_{d\le n,\, td^{r-1}\le 1} G_r\big(t\, d^{r-1}, {\textstyle\frac{n}{d}}\big)+\sum_{d\le n,\, td^{r-1}> 1} G_r(t\, d^{r-1}, {\textstyle\frac{n}{d}}\big):= \text{I}+\text{II}.
$$
For $td^{r-1}\le 1$, thanks to \eqref{eq:estimate for M, t<1}, we can write
$$
G_r\big(t\, d^{r-1}, {\textstyle\frac{n}{d}}\big)=\frac{1}{\zeta(r)} \,
\Omega_r(td^{r-1}) \,\Big(\frac{n}{d}\Big)^r+O_{t}\Big(\Big(\frac{n}{d}\Big)^{r-1}\, \Big)
$$
where we have used, in the $O$ term, that $td^{r-1}\ge t$ since $d\ge 1$. So the term I can be written
as
\begin{align*}
\text{I}&
= \frac{n^r}{\zeta(r)} \sum_{d\le n,\, td^{r-1}\le 1} \frac{\Omega_r(td^{r-1})}{d^r} +O_{t}\big(n^{r-1}\big)\,,
\end{align*}
where we have used that $r\ge 3$.

On the other hand, using now \eqref{eq:estimate for M, t>1} and $r \ge 3$,
\begin{align*}
\text{II}= \sum_{d\le n,\, td^{r-1}> 1} \frac{1}{\zeta(r)}\, \Big(\frac{n}{d}\Big)^r+O\Big(\sum_{d \le n}\Big(\frac{n}{d}\Big)^{r-1}\Big)
= \frac{n^r}{\zeta(r)} \sum_{d\le n,\, td^{r-1}> 1} \frac{1}{d^r} +O\big(n^{r-1})\,.
\end{align*}
Adding  up the expressions for I and II,
\begin{align*}
\sum_{d\le n} G_r&\,\big(t\, d^{r-1}, {\textstyle\frac{n}{d}}\big)
= \frac{n^r}{\zeta(r)} \sum_{d\le n} \frac{1}{d^r}+ \frac{n^r}{\zeta(r)} \sum_{d\le n,\, td^{r-1}\le  1} \frac{(\Omega_r(td^{r-1})-1)}{d^r}+O_{t}\big(n^{r-1}\big)
\\
&=n^r \Big(1-\frac{1}{\zeta(r)} \sum_{d\le n,\, td^{r-1}\le  1} \frac{(1-\Omega_r(td^{r-1}))}{d^r}\Big)+O_{t}\big(n^{r-1}\big)\,,
\end{align*}
from which inequality \eqref{eq:intro pdf lcm rge2-conzeta} is proved.

\subsubsection{Proof of \eqref{eq:intro pdf lcm rge2-conTr}}
We now obtain a lower  bound of the complementary probability by restricting to those $r$-tuples whose pairwise gcd are all equal and then partitioning according to the value of that common gcd.
Fix $0\le t\le 1$. Observe that
$$
\#\{\mathbf{x}\le n: \lcm(\mathbf{x})>tn^r\}
\ge
\sum_{k=1}^n  \#\{\mathbf{x}\le n: \lcm(\mathbf{x})>tn^r, \gcd(x_i,x_j)=k \text{ for $i\ne j$}\}\,.
$$

Writing $x_i=ky_i$, for each $i=1,\dots, r$, we get that  $\mathbf{y}\in \textrm{PC}$ and that $\lcm(\mathbf{x})=k\cdot y_1\cdots y_r$, and conversely. So we can write
\begin{align*}
&\frac{1}{n^r}\#\{\mathbf{x}\le n: \lcm(\mathbf{x})>tn^r\}
\ge
\frac{1}{n^r}\sum_{k=1}^n  \#\{\mathbf{y}\le {\textstyle\frac{n}{k}}: y_1\cdots y_r>tk^{r-1} {\textstyle(\frac{n}{k})^r}, \mathbf{y}\in \textrm{PC}\}
\\
&\qquad=
\sum_{tk^{r-1}<1}  \frac{1}{k^r}\ \frac{1}{(n/k)^r}\  \#\{\mathbf{y}\le {\textstyle\frac{n}{k}}: y_1\cdots y_r>tk^{r-1} {\textstyle(\frac{n}{k})^r}, \mathbf{y}\in \textrm{PC}\}\,.
\end{align*}

For fixed $t$, the sum has a finite number of terms, and the $k^{\text{th}}$ one tends to
$$
\frac{1}{k^r}T_r\,(1-\Omega_r(t k^{r-1})),\quad\text{as $n\to\infty$,}
$$
according to \eqref{eq:equidistribution of PC in Ar}. So the whole sum tends
$$
T_r\sum_{tk^{r-1}<1} \frac{1-\Omega_r(t k^{r-1})}{k^r},\quad\text{as $n\to\infty$,}
$$
and \eqref{eq:intro pdf lcm rge2-conTr} is proved.

\subsection{Moments of the lcm of $r$-tuples}
In this subsection, we shall prove Theorem \ref{teor:moments lcm-r}, with an argument akin to that in Theorem 10 of \cite{GS}.

\subsubsection{Proof of \eqref{eq:intro, moments lcm rge2-conzeta}}

The following lemma is  a direct application of~\eqref{eq:use of mu}:
\begin{lemma}
\label{lema:funcion fkr}Fix $q\ge 1$ and consider the summatory function $\acum{I_q}(n)=\sum_{j\le n} j^q$. Then, for $r\ge 2$,
\begin{equation}
\label{eq:estimacion fkr} 
\sum_{\substack{\mathbf{y}\le n,\\
\gcd(\mathbf{y})=1}} y_1^q\cdots y_r^q=\sum_{d\le n} \mu(d) \,d^{rq}\, \big[\acum{I_q}(n/d)\big]^r.
\end{equation}
\end{lemma}

Applying this lemma and the trivial estimate $\lcm(\mathbf{y})\le y_1 \cdots y_r$ we get
\begin{align}
\nonumber \sum_{\mathbf{x}\le n} \lcm(\mathbf{x})^q
&=\sum_{d\le n} \sum_{\substack{\mathbf{x}\le n,
\\
\gcd(\mathbf{x})=d}} \lcm(\mathbf{x})^q =\frac{1}{n^r}\sum_{d\le n} d^q \sum_{\substack{\mathbf{y}\le n/d,\\
\gcd(\mathbf{y})=1}} \lcm(\mathbf{y})^q
\\
\label{eq:lcm in terms of Sq}
&\le \sum_{kd\le n} d^q \,\mu(k) \,k^{rq}\, \Big[\acum{I_q}\Big(\frac{n}{dk}\Big)\Big]^r\, .
\end{align}

This means that
$$
\E\big(\lcm\big(X^{(n)}_1, \dots,X^{(n)}_r\big)^q\le \frac{1}{n^r}\sum_{kd\le n} d^q \,\mu(k) \,k^{rq}\, \Big[\acum{I_q}\Big(\frac{n}{dk}\Big)\Big]^r,
$$
and we will get \eqref{eq:intro, moments lcm rge2-conzeta} by estimating this last sum. From \eqref{eq:acumulador de fq}, we deduce that
$$
\Big[\acum{I_q}\Big(\frac{n}{dk}\Big)\Big]^r=\frac{1}{(q+1)^r}\, \Big(\frac{n}{dk}\Big)^{r(q+1)}+  O_{q,r}\Big(\frac{n}{dk}\Big)^{r(q+1)-1}.
$$
Plugging this into \eqref{eq:lcm in terms of Sq}, the sum of leading terms is given by
\begin{align*}
n^{rq}\, \frac{1}{(q+1)^r}\, \frac{\zeta(r(q+1)-q)}{\zeta(r)} +O_{q,r}(n^{r(q-1)+1}),
\end{align*}
while the sum of the error terms is seen to be $O_{q,r}\big(n^{rq-1}\big)$. Adding this up, we obtain \eqref{eq:intro, moments lcm rge2-conzeta}.

\subsubsection{Proof of \eqref{eq:intro, moments lcm rge2-conTr}} Arguing as in the proof of \eqref{eq:intro pdf lcm rge2-conTr},
\begin{align*}
&\frac{1}{n^{rq+r}} \sum_{\mathbf{x}\le n} \lcm(\mathbf{x})^q
\ge \frac{1}{n^{rq+r}} \sum_{k=1}^n \sum_{\substack{\mathbf{x}\le n\\ \gcd(x_i,x_j)=k,  i\ne j}} \lcm(\mathbf{x})^q
= \sum_{k=1}^n  \frac{k^q}{n^{rq+r}}
\sum_{\substack{\mathbf{y}\le n/k\\ \mathbf{y}\in\textrm{PC}}}
y_1^q\cdots y_r^q\,.
\end{align*}
According to \eqref{eq:equidistribution for moments} (with
$\beta_j=1/k$), for each summand we have that
$$
\lim_{n \to \infty}  \frac{k^q}{n^{rq+r}} \sum_{\substack{\mathbf{y}\le n/k\\ \mathbf{y}\in\textrm{PC}}} y_1^q\cdots y_r^q=
\frac{T_r}{(q+1)^r}\frac{1}{k^{r(q+1)-q}}
$$
From the bound
$$
\frac{k^q}{n^{rq+r}} \sum_{\substack{\mathbf{y}\le n/k\\
\mathbf{y}\in\textrm{PC}}} y_1^q\cdots y_r^q\le \frac{k^q}{n^{rq+r}}\,
\Big(\frac{n}{k}\Big)^{rq}\,
\Big(\frac{n}{k}\Big)^r=\frac{1}{k^{r(q+1)-q}},
$$
dominated convergence and the fact  that $\sum_{k=1}^\infty
1/{k^{r(q+1)-q}} < +\infty$, we deduce
$$\lim_{n \to \infty}\frac{1}{n^{rq+r}} \sum_{k=1}^n \sum_{\substack{\mathbf{x}\le n\\ \gcd(x_i,x_j)=k,  i\ne j}} \lcm(\mathbf{x})^q
=\frac{T_r}{(q+1)^r}\ \zeta(r(q+1)-q)\,,
$$
and therefore,
$$
\begin{aligned}
\liminf_{n \to \infty}\frac{1}{n^{rq}}\E\big(\lcm(X^{(n)}_1,\ldots,
X^{(n)}_r\big)&=\liminf_{n \to \infty}\frac{1}{n^{rq+r}} \sum_{\mathbf{x}\le n} \lcm(\mathbf{x})^q\ge\frac{T_r\, \zeta(r(q+1)-q)}{(q+1)^r}\
.
\end{aligned}$$ This proves \eqref{eq:intro, moments lcm
rge2-conTr}.

\subsection{The logarithm of the lcm}\label{subsection:log of lcm} The structure of the lcm exhibited in \eqref{eq:lcm as a product} invites to consider its logarithm, as it may be written in terms of sums of logarithms of gcd's of different lengths:
\begin{align}\label{eq:ln of lcm as a sum}
\ln\big(\lcm(a_1,\dots, a_r)\big)=&\,\sum_{j=1}^r \ln(a_j)- \sum_{i< j}\ln(\gcd(a_i,a_j))
\\ \nonumber
&+ \sum_{i< j< k}\ln(\gcd(a_i,a_j, a_k))-\cdots
\end{align}
Using this, we can prove the following:
\begin{prop}
\label{prop:moments of ln of lcm}For $r\ge 2$,
\begin{equation}
\label{eq:moments of ln of lcm}
\lim_{n\to\infty} \Big[\E\big(\ln(\lcm(X_1^{(n)},\dots, X_r^{(n)}))\big)-\frac{r}{n}\sum_{j=1}^n \ln(j)\Big]=\sum_{k=2}^r {r\choose k} (-1)^k \,\frac{\zeta'(k)}{\zeta(k)}\, .
\end{equation}
\end{prop}
\begin{proof}
Observe that, by Ces\`{a}ro's formula \eqref{eq:Cesaro sum},
\begin{align*}
\E\big(\ln(\lcm(X_1^{(n)},\dots, X_r^{(n)}))&\big)
= {r\choose 1} \E\big(\ln(X_1^{(n)})\big)- {r\choose 2} \E\big(\ln(\gcd(X_1^{(n)}, X_2^{(n)}))\big)+\cdots 
\\
&=\frac{r}{n} \sum_{j=1}^n \ln(j)-\sum_{k=2}^r {r\choose k} (-1)^k \sum_{j=1}^n (\mu*\ln)(j)\Big(\Big\lfloor\frac{n}{k}\Big\rfloor\, \frac{1}{n}\Big)^k
\\
&=\frac{r}{n} \sum_{j=1}^n \ln(j)-\sum_{k=2}^r {r\choose k} (-1)^k \sum_{j=1}^n \Lambda(j)\Big(\Big\lfloor\frac{n}{k}\Big\rfloor\, \frac{1}{n}\Big)^k\,.
\end{align*}
We have used the fact that $(\mu*\ln)(j)=\Lambda(j)$, where $\Lambda$ denotes the von Mangoldt's function (see Theorem 295 in \cite{HW}).  
For $k\ge 2$ fixed,
$$
\lim_{n\to\infty}\sum_{j=1}^n \Lambda(j)\Big(\Big\lfloor\frac{n}{j}\Big\rfloor\, \frac{1}{n}\Big)^k=\sum_{j=1}^\infty \frac{\Lambda(j)}{j^k}=-\frac{\zeta'(k)}{\zeta(k)}\,,
$$
where we have used the trivial estimate $\Lambda(n)\le \ln(n)$ and the standard expression of the Dirichlet series of the function~$\Lambda$ (see Theorem 294 in \cite{HW}).\end{proof}

By the way, from \eqref{eq:moments of ln of lcm} and Jensen's inequality one gets that
$$
\liminf_{n\to\infty}
\frac{\E\big(\lcm\big(X^{(n)}_1, \dots,X^{(n)}_r\big)^q\big)}{n^{rq}}\ge e^{q(H(r)-r)}\,,
$$
where $H(r)$ is the function defined in the right hand side of \eqref{eq:moments of ln of lcm}, a lower bound weaker than \eqref{eq:intro, moments lcm rge2-conTr}, but of the proper order.

\subsection{Alternative normalization of lcm}\label{subsection:alternatrive_normalization}

As we have mentioned before, one may alternatively normalize lcm by dividing it by the product of the numbers. With an argument similar, but simpler, to the one used to prove Theorem \ref{teor:pdf lcm-r}, one may derive:
\begin{prop}\label{prop:prob_lcm_divide_by_product_r}
For $r \ge 2$, and for every $0 < t \le 1$,
\begin{align*}
\liminf_{n\to\infty}\P\Big(\frac{\lcm\big(X^{(n)}_1, \dots,X^{(n)}_r\big)}{X^{(n)}_1\cdots X^{(n)}_r}\le t \Big)
&\ge 1-\frac{1}{\zeta(r)}
\sum_{j^{r-1} < 1/t}
\frac{1}{j^r}\,,
\\
\limsup_{n\to\infty}\P\Big(\frac{\lcm\big(X^{(n)}_1, \dots,X^{(n)}_r\big)}{X^{(n)}_1\cdots X^{(n)}_r}\le t \Big)
&\le 1-\ T_r
\ \sum_{j^{r-1} < 1/t}
\frac{1}{j^r}\,.
\end{align*}
Actually, for $r=2$, we have equality
 for every $0< t\le1:$
 $$
  \lim_{n\to\infty}\P\Big(\frac{\lcm\big(X^{(n)}_1, X^{(n)}_2\big)}{X^{(n)}_1 X^{(n)}_2}\le t \Big)
= 1- \frac{1}{\zeta(2)}
\ \sum_{j < 1/t}
\frac{1}{j^2}\,.
$$
\end{prop}
Of course, the statement above for $r=2$ gives the asymptotic behavior of the distribution function of $1/\gcd(X^{(n)}_1, X^{(n)}_2)$. The limiting distribution is discrete: it assigns mass $\frac{1}{\zeta(2)k^2}$ to the point ${1}/{k}$, for every integer $k \ge 1$; in contrast to the  limit distribution in
Theorem \ref{theor:lcm of pairs}, part (a), which has no point masses.

For moments, we have:
\begin{prop}\label{prop:moments_lcm_divide_by_product_r}
For $r \ge 2$, and integer $q \ge 1$,
$$
\begin{aligned}
\limsup_{n \to \infty} \E\Big(\Big(\frac{\lcm\big(X^{(n)}_1, \dots,X^{(n)}_r\big)}{X^{(n)}_1\cdots X^{(n)}_r}\Big)^q\Big) &\le \frac{1}{\zeta(r)}\,\zeta\big(r(q+1)-q\big)\,,
\\
\liminf_{n \to \infty} \E\Big(\Big(\frac{\lcm\big(X^{(n)}_1, \dots,X^{(n)}_r\big)}{X^{(n)}_1\cdots X^{(n)}_r}\Big)^q\Big) &\ge \ T_r \ {\zeta\big(r(q+1)-q\big)}\,.
\end{aligned}
$$
Actually, for $r=2$, we have equality for every integer $q\ge 1:$
$$
\lim_{n \to \infty} \E\Big(\Big(\frac{\lcm\big(X^{(n)}_1, X^{(n)}_2\big)}{X^{(n)}_1\cdot X^{(n)}_2}\Big)^q\Big) = \frac{1}{\zeta(2)}\,\zeta\big(q+2\big)\,.
$$
\end{prop}

\section{The case $r=3$}\label{lcm for r=3}

The analysis of this case $r=3$ rests on a particular partition of the triples of integers which is based in the following factorization lemma:
\begin{lemma}\label{lemma:triple_representation}
To any triple of  integers $(x,y,z)$ we may assign uniquely  an integer~$D$ and triples of integers $(a,b,c)\in \text{\upshape PC}$ and $(u,v,w)\in \text{\upshape PC}_{(c,b,a)}$ so that
$$
x=D \, (ab)\, u,\quad
y=D\,(ac)\,v,\quad
z=D \,(bc)\, w.
$$
In fact, $
D=\gcd(x,y,z)
$,
\begin{equation}
\begin{cases}
a=\gcd(x,y)/D,\\
b=\gcd(x,z)/D,\\
c=\gcd(y,z)/D,
\end{cases} \quad\text{and}\quad
\begin{cases}
u={x}/{(Dab)},\\
v={y}/{(Dac)},\\
w={z}/{(Dbc)}.
\end{cases}
\end{equation}
Moreover,
$$
\lcm(x,y,z)=D (abc)(uvw)\,.
$$
\end{lemma}
The proof is direct; we just insist that $(u,c), (v,b), (w,a)$ are required to be coprime couples. For pairs of integers $(x,y)$ the analogous representation is $x=Du, y=Dv$, with $D=\gcd(x,y)$ and $u,v$ coprime.

\subsection{Proof of part a) of Theorem \ref{teor:lcm-rigual3}}

Fix an integer $n\ge 1$ and $0 <t \le 1$. We follow Lemma \ref{lemma:triple_representation} and partition the required counting:
\begin{align*}
\#\{1 \le &\,x,y,z \le n\,; \lcm(x,y,z)>t n^3\}
\\
&=
\sum_{D=1}^\infty \sum_{(a,b,c) \in \textrm{PC}} \#\Bigg\{\ \begin{aligned}&u\le n/(Dab)\\ &v\le n/(Dac)\\&w\le n/(Dbc)\end{aligned} : (u,v,w)\in\textrm{PC}_{(c,b,a)}, \
 \ uvw > t\frac{n^3}{Dabc}
\Bigg\}
\end{align*}
Now, according to \eqref{eq:equidistribution of PCE in Ar}, the argument of this double sum, for fixed $D$ and fixed $a,b,c$, is asymptotically
$$
\sim n^3 \, T_3 \, \Upsilon_3(abc) \, \frac{1}{D^3 (abc)^2}\,\Big(1-\Omega_3(tD^2 abc)\Big)\,
$$
and is bounded by
$$
\le n^3 \, \frac{1}{D^3 (abc)^2}\,
$$
Since
$$
\sum_{D=1}^\infty \sum_{(a,b,c) \in \textrm{PC}}\frac{1}{D^3 (abc)^2}=\zeta(3)\sum_{(a,b,c) \in \textrm{PC}}\frac{1}{(abc)^2}=\zeta(3) \sum_{m=1}^\infty \frac{3^{\omega(m)}}{m^2} < +\infty\,,
$$
dominated convergence gives the result. Here we have used that for any given integer $m$ there are exactly $3^{\omega(m)}$ triples $(a,b,c) \in \textrm{PC}$ such that $m=abc$.

\begin{remark}\label{remark:check T(2)}
{\upshape For $t=0$, equation \eqref{eq:pdf lcm-rigual3} reads $T_3\,\zeta(3)\,J(2) =1$. This begs to be proved directly. Observe that the function $\hat\Upsilon_3(m)=\Upsilon_3(m)3^{\omega(m)}$ is multiplicative, and that, for any prime $p$ and any positive integer $a$, $\hat\Upsilon_3(p^a)=\hat\Upsilon_3(p)=3\frac{1+1/p}{1+2/p}$. We can write the Dirichlet series defined in \eqref{eq:definition of T(s)} as a product over primes:
$$
J(s)=\sum_{m=1}^\infty \frac{\tilde\Upsilon_3(m)}{m^s}=\prod_p \Big(1+\frac{\tilde\Upsilon_3(p)}{p^s}+\frac{\tilde\Upsilon_3(p^2)}{p^{2s}}+\cdots)=\prod_p \Big(1+\frac{3(p+1)}{(p+2)(p^s-1)}\Big),
$$
so
$$
J(2)=\prod_p \Big(1+\frac{3}{(p+2)(p-1)}\Big)=\prod_p \Big(\frac{p^2+p+1}{(p+2)(p-1)}\Big).
$$
The reader may check, from the definition of $T_3$ in \eqref{eq:definition of Tr} and the Euler product expression for $\zeta(s)=\prod_{p}\frac{1}{1-p^{-s}}$, that $T_3\,\zeta(3)=\prod_p\frac{(p-1)(p+2)}{p^2+p+1}$.}
\end{remark}

\subsection{Proof of part b) of Theorem \ref{teor:lcm-rigual3}}

Fix $q\ge 1$. Lemma \ref{lemma:triple_representation} allow us to write
$$
\sum_{x,y,z \le n} \lcm(x,y,z)^q=\sum_{D=1}^\infty \sum_{(a,b,c)\in \textrm{PC}} \Big(D^q (abc)^q {\sum}^{\prime}_{D;\, a,b,c} (uvw)^q\Big)\,,
$$
where for $D$ and $(a,b,c)$ fixed, the corresponding sum ${\sum}^{\prime}_{D;\, a,b,c}$ extends over
$$
\big\{(u,v,w) \in \textrm{PC}, \ u \le n/[D(ab)],\
v \le n/[D(ac)], \
w \le n/[D(bc)]
\big\}.
$$

By \eqref{eq:equidistribution for moments PCE}, each  ${\sum}^{\prime}_{D;\, a,b,c}$ is seen to be, asymptotically
$$
{\sum}^{\prime}_{D;\, a,b,c} \ \sim n^{3q+3}\, T_3 \ \Upsilon_3(abc) \frac{1}{(q+1)^3} \Big(\frac{1}{D^3 (abc)^2}\Big)^{q+1}, \quad \text{as} \ n \to \infty\,,
$$
and is bounded by
$$
{\sum}^{\prime}_{D;\, a,b,c}  \ \le n^{3q+3} \Big(\frac{1}{D^3 (abc)^2}\Big)^{q+1}\, .
$$
Since
$$
\sum_{D=1}^\infty \sum_{(a,b,c)\in \textrm{PC}} D^q (abc)^q \Big(\frac{1}{D^3 (abc)^2}\Big)^{q+1}=\zeta(2q+3)\sum_{m=1}^\infty \frac{3^{\omega(m)}}{m^{q+2}} <\infty,
$$
dominated convergence gives that
$$
\lim_{n \to \infty}\frac{1}{n^{3q+3}} \sum_{x,y,z \le n} \lcm(x,y,z)^q=T_3 \frac{1}{(q+1)^3}\zeta(2q+3)\sum_{m=1}^\infty \frac{\Upsilon_3(m)\, 3^{\omega(m)}}{m^{q+2}}
,
$$
as claimed in \eqref{eq:moments lcm-rigual3}.

\subsection{Case $r=3$, with alternative normalization}

We have for every $0 < t \le 1$,
$$
\lim_{n \to \infty} \P \Big(\Big(\frac{\lcm\big(X^{(n)}_1, X^{(n)}_2,X^{(n)}_3\big)}{X^{(n)}_1\ X^{(n)}_2\ X^{(n)}_3}\Big) \le t\Big)=T_3 \sum_{D^2 m \ge 1/t} \frac{1}{D^3} \frac{\Upsilon_3(m) 3^{\omega(m)}}{m^2}\, .
$$
Also, for integer $q \ge 1$:
$$
\lim_{n \to \infty}  \E \Big(\Big(\frac{\lcm\big(X^{(n)}_1, X^{(n)}_2,X^{(n)}_3\big)}{X^{(n)}_1\ X^{(n)}_2\ X^{(n)}_3}\Big)^q\Big)=T_3 \ \zeta(2q+3) J(2+q)\, .
$$

\section{Waiting times}\label{sec:waiting times}
Fix $n\ge 2$ and consider the following experiment: draw integers
$x_1,x_2,\dots$ uniformly and independently from $\{1,\dots,n\}$,
and calculate the sequences of successive gcds and lcms:
\begin{align*}
&z_1=x_1,\ z_2=\gcd(x_1,x_2),\ z_3=\gcd(x_1,x_2,x_3),\dots
\\
&w_1=x_1,\ w_2=\lcm(x_1,x_2),\ w_3=\lcm(x_1,x_2,x_3),\dots
\end{align*}
The sequence $(z_j)$ is decreasing, and each $z_j\ge 1$, while the
sequence $(w_j)$ is increasing and each
$w_j \le \lcm(1,\dots,n)$.

For each of these random sequences we are interested in the random
variable that registers the first time when they reach their
respective limiting values. Again, the case of the lcm is quite more involved than the case of the gcd.

\subsection{Waiting time for the gcd}\label{subsec:wt for gcd}
For fixed $n$, consider the (decreasing) sequence $(\mathcal{Z}_m^{(n)})$ of random variables given~by
$$
\mathcal{Z}_1^{(n)}=X_1^{(n)},\quad \mathcal{Z}_m^{(n)}=\gcd(\mathcal{Z}_{m-1}^{(n)}, X_m^{(n)})=\gcd(X_1^{(n)},\dots, X_m^{(n)})\quad\text{for $m\ge 2$;}
$$

Denote by $\mathcal{T}_n$ the first time when the sequence $(\mathcal{Z}_m^{(n)})$ reaches the value $1$. The variable $\mathcal{T}_n$ takes values $1,2,\dots$
\begin{lemma}
\label{lemma:waiting time for gcd}
{\rm a)} For fixed $n$,
$
\P(\mathcal{Z}_m^{(n)}=1)\to 1$ as $m\to\infty$.

\smallskip
\noindent {\rm b)} The mass function of $\mathcal{T}_n$ is given by
\begin{equation}\label{eq:dist of Tn gcd}
\P(\mathcal{T}_n>m)=-\sum_{k=2}^n \mu(k) \Big(\frac{1}{n}\Big\lfloor \frac{n}{k}\Big\rfloor\Big)^m\quad\text{for $m\ge 1$}.
\end{equation}
and, for each $m\ge 1$,
\begin{equation}\label{eq:dist of T gcd}
\lim_{n\to\infty}\P(\mathcal{T}_n> m)=1-\frac{1}{\zeta(m)}.
\end{equation}
\end{lemma}
Notice that the case $m=1$ of \eqref{eq:dist of Tn gcd} reduces to
$\P(\mathcal{T}_n>1)=1-{1}/{n}$. The case $m=1$ of
\eqref{eq:dist of T gcd} is then obvious.
\begin{proof}
a) Observe that $\P(\mathcal{Z}_1^{(n)}=1)=1/n$. Recall from \eqref{eq:prob gcd=1 in terms of Cesaro} that, for $m\ge 2$,
$$
\P(\mathcal{Z}_m^{(n)}=1)=\P(\gcd(X_1^{(n)},\dots, X_m^{(n)})=1)=\frac{1}{n^m}\sum_{k=1}^n \mu(k)\, \Big\lfloor \frac{n}{k}\Big\rfloor^m ,
$$
so
$$
|1-\P(\mathcal{Z}_m^{(n)}=1)|\le \sum_{k=2}^\infty \frac{1}{k^m}\le \frac{3}{2^m}\,
$$
and therefore,
$\P(\mathcal{Z}_m^{(n)}=1)\to 1 $ as $m\to\infty$ (for fixed $n$).

\newpage

b) Observe that, for each $m\ge 1$, the events $\{\mathcal{T}_n\le m\}$ and $\{\mathcal{Z}_m^{(n)}=1\}$ coincide,
 and therefore,
\begin{align*}
\P\{\mathcal{T}_n> m\}&=1-\P\{\mathcal{T}_n\le m\}=1-
\frac{1}{n^m}\sum_{k=1}^n \mu(k)\, \Big\lfloor \frac{n}{k}\Big\rfloor^m =-\sum_{k=2}^n \mu(k)\, \Big(\frac{1}{n}\Big\lfloor \frac{n}{k}\Big\rfloor\Big)^m.
\end{align*}
From here one deduces immediately that, for $m \ge 2$,
$$\Big|\P\big(\mathcal{T}_n>
m\big)-\Big(1-\frac{1}{\zeta(m)}\Big)\Big|=O\Big(\frac{\ln(n)}{n}\Big)\,
.
$$ This gives \eqref{eq:dist of T gcd} for each $m\ge 2$.
\end{proof}
The above result tells us that $\mathcal{T}_n$ converges in
distribution to $\mathcal{T}$, where $\mathcal{T}$ is the random
variable given by $\P(\mathcal{T}\le m)=1/\zeta(m)$. Formulas for
the expected values of these variables,~$\mathcal{T}_n$
and~$\mathcal{T}$, are given in the following result:

\begin{theorem}
\label{teor:mean of waiting time for gcd}
\begin{equation}\label{eq:mean T lcm}
\lim_{n\to\infty}\E(\mathcal{T}_n)=1-\sum_{k=2}^\infty \frac{\mu(k)}{k-1}=\E(\mathcal{T})=2+\sum_{m=2}^\infty \Big(1-\frac{1}{\zeta(m)}\Big) \,.
\end{equation}
\end{theorem}
The numerical value of $\E(\mathcal{T})$ is around 2,7052; by the
way, the sequence $\E(\mathcal{T}_n)$ is not increasing.
\begin{proof} Changing the summation order,
\begin{align*}
\E(\mathcal{T}_n)&=\sum_{m=0}^\infty \P(\mathcal{T}_n>m)=2-\frac{1}{n}-\sum_{m=2}^\infty \sum_{k=2}^n \mu(k) \Big(\frac{1}{n}\Big\lfloor \frac{n}{k}\Big\rfloor\Big)^m
\\
&=2-\frac{1}{n}- \sum_{k=2}^n \mu(k) \sum_{m=2}^\infty\Big(\frac{1}{n}\Big\lfloor \frac{n}{k}\Big\rfloor\Big)^m=2-\frac{1}{n}-\sum_{k=2}^n \mu(k) \frac{(\frac{1}{n}\lfloor\frac{n}{k}\rfloor)^2}{1-\frac{1}{n}\lfloor\frac{n}{k}\rfloor}.
\end{align*}
Now observe that
$$
\frac{(\frac{1}{n}\lfloor\frac{n}{k}\rfloor)^2}{1-\frac{1}{n}\lfloor\frac{n}{k}\rfloor}\le \frac{1}{k(k-1)}\quad\text{and}\quad \lim_{n\to\infty}\frac{(\frac{1}{n}\lfloor\frac{n}{k}\rfloor)^2}{1-\frac{1}{n}\lfloor\frac{n}{k}\rfloor}= \frac{1}{k(k-1)}.
$$
As $\sum_{k\ge 2} 1/(k(k-1))=1$, by dominated convergence,
$$
\lim_{n\to\infty} \E(\mathcal{T}_n)=2-\sum_{k=2}^\infty \frac{\mu(k)}{k(k-1)}=2-\sum_{k=2}^\infty \mu(k) \Big(\frac{1}{k-1}-\frac{1}{k}\Big)=1-\sum_{k=2}^\infty \frac{\mu(k)}{k-1}\,,
$$
where we have used Landau's classical result (see \cite{Landau}) that $\sum_{k=1}^\infty \mu(k)/k=0$. It can be seen seen (again by dominated convergence) that this sum coincides with
$$
\E(\mathcal{T})=2+\sum_{m=2}^\infty \Big(1-\frac{1}{\zeta(m)}\Big).\qedhere
$$
\end{proof}

\subsection{Waiting for the lcm}\label{subsec:wt for lcm}
Consider now the (increasing) sequence $(\mathcal{W}_m^{(n)})$ of random variables given~by
$$
\mathcal{W}_1^{(n)}=X_1^{(n)},\quad \mathcal{W}_m^{(n)}=\lcm(\mathcal{W}_{m-1}^{(n)}, X_m^{(n)})=\lcm(X_1^{(n)},\dots, X_m^{(n)})\quad\text{for $m\ge 2$;}
$$
Let us denote again by $\mathcal{T}_n$ the first time when the sequence $(\mathcal{W}_m^{(n)})$ reaches its limiting value, $\lcm(1,\dots,n)$.

\smallskip
In the analysis of $\mathcal{T}_n$ we will use the following notation: for each prime $p\le n$, denote by $\gamma_p(n)$ the integer such that
\begin{equation}
\label{eq:def of gammap(n)}
p^{\gamma_p(n)}\le n \quad\text{and}\quad p^{\gamma_p(n)+1} >n.
\end{equation}
In formula, $\gamma_p(n)=\lfloor \ln n/\ln p\rfloor$. Notice that
$$
\lcm(1,\dots,n)=\prod_{p\le n} p^{\gamma_p(n)}.
$$

Now define $\beta_p(n)$ as the positive integer satisfying
\begin{equation}
\label{eq:def of betap(n)}
p^{\gamma_p(n)}\, \beta_p(n)\le n \quad\text{and}\quad p^{\gamma_p(n)}\, (\beta_p(n)+1) >n,
\end{equation}
so that $\beta_p(n)=\lfloor n/p^{\gamma_p(n)}\rfloor$. Observe that $1\le \beta_p(n)<p$.

For each prime $p\le n$, we consider the set $\mathcal{C}_p(n)$ of integers given by
$$
\mathcal{C}_p(n)=\{p^{\gamma_p(n)}, 2p^{\gamma_p(n)},\dots,\beta_p(n)\, p^{\gamma_p(n)}\}.
$$
It is immediate to check that, for fixed $n$,  the sets $\mathcal{C}_p(n)$ and $\mathcal{C}_q(n)$ are disjoint if $p$ and $q$ are different primes. Finally, call
$$
\mathcal{C}^*(n)=\{1,\dots,n\}-\bigcup_{p\le n} \mathcal{C}_p(n).
$$

Observe that the event of interest, $\{\lcm(X_1^{(n)},\dots, X_m^{(n)})=\lcm(1,\dots,n)\}$,
can be written as
$$
\{\text{for each $p\le n$, at least one among
$X_1^{(n)},\dots,X_m^{(n)}$ belongs to $\mathcal{C}_p(n)$}\}.
$$

This observation allows us to rewrite the waiting time question as a ``weighted''
coupon collector problem: we draw coupons (independently) from
$\{1,\dots,n\}$, and to complete the collection of interest means to get, at
least, one coupon from each of the classes $\mathcal{C}_p(n)$, $p\le
n$. There are $\pi(n)$ different classes, where $\pi(n)$ denotes the
number of primes $\le n$, each one of them having ``weight''
$\beta_p(n)/n$. The coupons from the set $\mathcal{C}^*(n)$ are
useless for our objective. See, for instance, \cite{FGT}, \cite{BP},
or the survey \cite{BH} for information about a variety of coupon
collector problems.

In this language, the variable $\mathcal{T}_n$ registers
the time when the coupon collection is completed. Observe that $\P(\mathcal{T}_n>l)=1$ if $l<\pi(n)$.
In general,
$$
\{\mathcal{T}_n>l\}=\bigcup_{p\le n} A_p(l)\,,
$$
where $A_p(l)$ is the event in which, among the first $l$ coupons
drawn, none of them belongs to $\mathcal{C}_p(n)$. Applying the
inclusion/exclusion principle, we can write
\begin{align*}
\P(\mathcal{T}_n>l)&=\sum_{p\le n} \P(A_p(l))-\sum_{p< q\le n} \P(A_p(l)\cap A_q(l))+ \cdots
\\
&=\sum_{p\le n} \Big(1-\frac{\beta_p(n)}{n}\Big)^l- \sum_{p<q\le n} \Big(1-\frac{\beta_p(n)+\beta_q(n)}{n}\Big)^l+\cdots
\end{align*}

From this expression for the distribution function of $\mathcal{T}_n$, and following \cite{FGT}, we get a compact formula for the expected  waiting time: $\mathcal{T}_n$:
\begin{lemma}
\label{lemma:formula for E(Tn) lcm}For fixed $n$,
\begin{equation}
\label{eq:formula for E(Tn) lcm}\E(\mathcal{T}_n)=n\int_0^\infty \Big[1-\prod_{p\le n} (1-e^{-t\, \beta_p(n)})\Big]\, dt.
\end{equation}
\end{lemma}
\begin{proof}
Write
\begin{align*}
\E(\mathcal{T}_n)&=\sum_{l=0}^\infty \P(\mathcal{T}_n>l)=\sum_{l=0}^\infty  \Big(\sum_{p\le n}\Big(1-\frac{\beta_p(n)}{n}\Big)^l- \sum_{p\ne p} \Big(1-\frac{\beta_p(n)+\beta_q(n)}{n}\Big)^l+\cdots\Big)
\\
&=\sum_{p\le n} \sum_{l=0}^\infty \Big(1-\frac{\beta_p(n)}{n}\Big)^l-\sum_{p<q \le n} \sum_{l=0}^\infty \Big(1-\frac{\beta_p(n)+\beta_q(n)}{n}\Big)^l+\cdots
\\
&=n\Big(\sum_{p\le n} \frac{1}{\beta_p(n)}-\sum_{p< q \le n} \frac{1}{\beta_p(n)+\beta_q(n)}+\cdots\Big).
\end{align*}
The identity \eqref{eq:formula for E(Tn) lcm} now follows.
\end{proof}

 If we were to care just for the specific coupons $\{p^{\gamma_p(n)}\}_{p\le n}$ (the first coupon in each class $\mathcal{C}_p(n)$), then the time $\widetilde{\mathcal{T}}_n$ to collect all of them will be, of course, longer than the time $\mathcal{T}_n$; and, in particular, on average: $\E\big(\widetilde{\mathcal{T}}_n\big)\ge \E\big(\mathcal{T}_n\big)$  (see later \eqref{eq:comparing_con_tiempotilde} for a precise comparison).  In this case there are exactly $\pi(n)$ coupons of interest (with weight 1) out of a total on $n$ coupons and, therefore, see Lemma \ref{lemma:harmonic_numbers},
 $$
 \E\big(\widetilde{\mathcal{T}}_n\big)=n\int_0^\infty \Big[1-\prod_{p\le n} (1-e^{-t})\Big]\, dt=n\int_0^\infty 1- (1-e^{-t})^{\pi(n)}\, dt=n H_{\pi(n)}\, .
 $$
 The asymptotic size of $\E\big(\widetilde{\mathcal{T}}_n\big)$ may be obtained by appealing to the elementary Lemma \ref{lemma:harmonic_numbers} and to the  standard error bound on the Prime Number Theorem:
 \begin{equation}
 \label{eq:bound_PNT}\Big|\pi(n)-\frac{n}{\ln(n)}\Big|\le C \frac{n}{\ln(n)^2}\, ,
 \end{equation}
  valid for each $n \ge 1$ ($C$ is an absolute constant). We may write:
 \begin{equation}\label{eq:estimate_tiempo_tilde}
 \begin{aligned}
 \E\big(\widetilde{\mathcal{T}}_n\big)=n H_{\pi(n)}&= n \big(\ln(\pi(n))+\gamma+O({\ln(n)}/{n})\big)\\
 &=n\, \ln(n)-n \ln\ln(n) +n \gamma+O\big(n/\ln(n)\big)\,.
 \end{aligned} \end{equation}

We would like to obtain an asymptotic expression like \eqref{eq:estimate_tiempo_tilde}, but for $\mathcal{T}_n$. For that purpose, we introduce the following frequency counting functions of the $\beta_p(n)$:
\begin{equation}
\label{eq:def of omegas}\omega_j(n)=\#\{p\le n: \beta_p(n)=j\}\, , \quad j \ge 1\, .
\end{equation}

A few properties of these $\omega_j$ are in order.
\begin{lemma}\label{lemma:properties_omegas}
\begin{itemize}
\item[a)] $\omega_j(n)=0$, if $j\ge \sqrt{n}$.\\
\item[b)] $\sum_{j=1}^\infty \omega_j(n)=\pi(n)$.\\
\item[c)] For $j \ge 1$ such that $j+1\le\sqrt{n}$,
\begin{equation}
\label{eq:bound for omegaj}
\pi\Big(\frac{n}{j}\Big)-\pi\Big(\frac{n}{j+1}\Big)\le \omega_j(n)\le \pi\Big(\frac{n}{j}\Big)-\pi\Big(\frac{n}{j+1}\Big)+\pi(\sqrt{n}).
\end{equation}
\item[d)] $\sum_{j=1}^\infty j \, \omega_j(n)=\sum_{p \le n} \beta_p(n) \sim n \ln(2)$, as $n \to \infty$. Actually,
\begin{equation}
\label{eq:average_betas}\frac{1}{n}\sum_{p\le n}\beta_p(n)=\ln(2)+O\Big(\frac{1}{\ln(n)}\Big)\, , \quad \text{as} \quad n \to \infty\, .
\end{equation}
\item[e)] For each $j \ge 1$,
\begin{equation}
\label{eq:limit for omegaj}\lim_{n\to\infty} \frac{\omega_j(n)}{\pi(n)}=\frac{1}{j(j+1)}.
\end{equation}
\end{itemize}
\end{lemma}

Observe that part d) means that asymptotically the proportion of useless coupons out of the total of $n$ coupons (the set $\mathcal{C}^*_n$) is  $1-\ln(2)$, about $30\%$.
\begin{proof} Statement a) is equivalent to $\beta_p(n)< \sqrt{n}$. To verify this, observe that from~\eqref{eq:def of gammap(n)} we deduce that $\gamma_p(n)=1$ if $\sqrt{n}<p\le n$. In this range, if $\beta_p(n)\ge\sqrt{n}$, we would get that $p\beta_p(n)>n$, a contradiction with~\eqref{eq:def of betap(n)}.
Whenever $\gamma_p(n)=\alpha\ge 2$, we have that $\beta_p(n)<p\le n^{1/\alpha}\le \sqrt{n}$.

Claim b) is immediate.

c) Let us introduce $J=J_n=\lfloor \sqrt{n}\rfloor-1$ (so that $J+1 \le \sqrt{n}$). For $1 \le j \le J$,  primes $p\le n$ such that $$
\frac{n}{j+1}<p\le \frac{n}{j}
$$
satisfy $\gamma_p(n)=1$ and $\beta_p(n)=j$; all other primes $p\le n$ not included among these~$J$ classes satisfy $p \le \frac{n}{J+1}\le \sqrt{n}$. Thus \eqref{eq:bound for omegaj} follows.

For d), we write
\begin{align*}
\sum_{j=1}^\infty j \, \omega_j(n)&=\sum_{p \le n} \beta_p(n) =\sum_{j=1}^J j \, \Big[\pi\Big(\frac{n}{j}\Big)-\pi\big(\frac{n}{j+1}\big)\Big]+\sum_{p\le \frac{n}{J+1}}\beta_p(n)\\&=\sum_{j=1}^J \pi\Big(\frac{n}{j}\Big)-J\cdot\pi\Big(\frac{n}{J+1}\Big)+\sum_{p\le \frac{n}{J+1}} \beta_p(n)
=
\sum_{j=1}^J \pi\Big(\frac{n}{j}\Big)+O\Big(\frac{n}{\ln(n)}\Big)\, ,
\end{align*}
where we have used summation by parts and the bounds
$$
\sum_{p\le {n}/{(J+1)}} \beta_p(n)\le \frac{n}{J+1} \,\pi\Big(\frac{n}{J+1}\Big)=O\Big(\frac{n}{\ln(n)}\Big) \quad \text{and} \quad
J\pi\Big(\frac{n}{J+1}\Big)=O\Big(\frac{n}{\ln(n)}\Big)\, .$$
Now, from the error bound \eqref{eq:bound_PNT}, we obtain
$$
\sum_{j=1}^J \pi\Big(\frac{n}{j}\Big)=n \sum_{j=1}^J \frac{1}{j \ln(n/j)}+O\Big(\frac{n}{\ln(n)}\Big)\, ,
$$
and, since,
$$
\sum_{j=1}^J \frac{1}{j \ln(n/j)} =\underset{=\ln(2)}{\underbrace{\int_1^{\sqrt{n}} \frac{1}{x \ln(n/x)}\, dx}}+O\Big(\frac{1}{\ln(n)}\Big),
$$
equation \eqref{eq:average_betas} follows.

Finally, \eqref{eq:limit for omegaj} of statement e) follows by dividing~\eqref{eq:bound for omegaj} by~$\pi(n)$ and invoking the Prime Number Theorem.
\end{proof}

The following lemma furnishes some precise asymptotic estimate for the frequency $\omega_1$.
\begin{lemma}
\label{lemma:asymptotics_omega1}
$$
\omega_1(n)=\frac{1}{2}\frac{n}{\ln(n)}+O\Big(\frac{n}{\ln^2(n)}\Big)
$$
\end{lemma}
\begin{proof} The error bound \eqref{eq:bound_PNT} readily gives that
$$
\pi(n)-\pi\Big(\frac{n}{2}\Big)=\frac{1}{2}\frac{n}{\ln(n)}+O\Big(\frac{n}{\ln^2(n)}\Big)\, ;
$$
the result follows from \eqref{eq:bound for omegaj} (for $j=1$) and the bound $\pi(\sqrt{n})=O\Big(\frac{\sqrt{n}}{\ln(n)}
\Big)$.
\end{proof}

We are now ready to estimate $\E(\mathcal{T}_n)$. We start by rewritting formula \eqref{eq:formula for E(Tn) lcm} as:
\begin{equation}
\label{eq:formula2 for E(Tn) lcm}
\E(\mathcal{T}_n)=n\int_0^\infty \Big[1-\prod_{j< \sqrt{n}} (1-e^{-tj})^{\omega_j(n)})\big]\, dt.
\end{equation}

By keeping just the factor corresponding to $j=1$ in \eqref{eq:formula2 for E(Tn) lcm}, we obtain the following lower bound:
\begin{equation}
\label{eq:lower bound of E(Tn) lcm}
\E(\mathcal{T}_n)\ge n\int_0^\infty \big[1-(1-e^{-t})^{\omega_1(n)}\big]\, dt=n\, H_{\omega_1(n)}\,,
\end{equation}
where we have resorted to Lemma \ref{lemma:harmonic_numbers}.

For an upper bound: using that $e^{-jt}\le e^{-2t}$ for $j\ge 2$, that $\sum \omega_j(n)=\pi(n)$ and the identity $1-xy=1-x+x(1-y)$, we may bound
\begin{align*}
&\E(\mathcal{T}_n)\le n\int_0^\infty \big[1-(1-e^{-t})^{\omega_1(n)}\, (1-e^{-2t})^{\pi(n)-\omega_1(n)}\big]\, dt
\\
&  \ =n\int_0^\infty \big[1-(1-e^{-t})^{\omega_1(n)}\big]\, dt+ n\int_0^\infty (1-e^{-t})^{\omega_1(n)}\, \big[1-(1-e^{-2t})^{\pi(n)-\omega_1(n)}\big]\, dt
\\
& \    =n\,H_{\omega_1(n)}+ n\int_0^\infty (1-e^{-t})^{\omega_1(n)}\, \big[1-(1-e^{-2t})^{\pi(n)-\omega_1(n)}\big]\, dt.
\end{align*}
Now, since $1-x^\delta \le \delta (1-x)$, for $0\le x\le 1$ and $\delta >0$, we may further bound
$$
\begin{aligned}
\E(\mathcal{T}_n)&\le n\,H_{\omega_1(n)}+ n\big(\pi(n)-\omega_1(n)\big)\int_0^\infty (1-e^{-t})^{\omega_1(n)}\, e^{-2t}\, dt\\
&=n\,H_{\omega_1(n)}+ n\frac{\pi(n)-\omega_1(n)}{(\omega_1(n)+1)(\omega_1(n)+2)}
\le n\,H_{\omega_1(n)}+ n\frac{\pi(n)-\omega_1(n)}{\omega_1(n)^2}
\end{aligned}
$$

We have proved:
\begin{theorem}
\label{theor:estimate for E(Tn) lcm with omega}The mean value of $\mathcal{T}_n$ satisfies:
\begin{equation}
\label{eq:estimate for E(Tn) lcm with omega}
n\,H_{\omega_1(n)}\le \E(\mathcal{T}_n)\le n\,H_{\omega_1(n)}+n\, \frac{\pi(n)-\omega_1(n)}{\omega_1(n)^2}.
\end{equation}
\end{theorem}

Finally,
\begin{corollary}\label{cor:asymptotics_waitLCM}
$$
\E(\mathcal{T}_n)=n\ln(n)-n \ln\ln(n)+n \big(\gamma-\ln(2)\big)+O\big(n/\ln(n)\big)\,.
$$
\end{corollary}
\begin{proof}
Using Lemma \ref{lemma:harmonic_numbers}, the limit \eqref{eq:limit for omegaj}, for $j=1$, and Lemma \ref{lemma:asymptotics_omega1}, we have that
\begin{align*}
&H_{\omega_1(n)}=\ln(\omega_1(n))+\gamma+O\Big(\frac{\ln(n)}{n}\Big)\\
&\qquad=\ln\Big(\frac{1}{2}\frac{n}{\ln(n)}\Big)+\gamma+O\Big(\frac{1}{\ln(n)}\Big)+O\Big(\frac{\ln(n)}{n}\Big)
=\ln\Big(\frac{1}{2}\frac{n}{\ln(n)}\Big)+\gamma+O\Big(\frac{1}{\ln(n)}\Big)\, .
\end{align*}
Also, from the limit \eqref{eq:limit for omegaj}, for $j=1$, we deduce
$$
\frac{\pi(n)-\omega_1(n)}{\omega_1(n)^2}=O\Big(\frac{\ln(n)}{n}\Big)\, .
$$
Combining these two estimates, we get the result.
\end{proof}

Observe that, as a consequence of Corollary \ref{cor:asymptotics_waitLCM} and \eqref{eq:estimate_tiempo_tilde}, we deduce that
\begin{equation}
\label{eq:comparing_con_tiempotilde}
\E(\widetilde{\mathcal{T}}_n)-\E(\mathcal{T})=n \ln(2)+O\Big(\frac{n}{\ln(n)}\Big)\, .
\end{equation}


\begin{thebibliography}{99}



\bibitem{Ap} \textsc{Apostol, T.\,M.}: \textit{Introduction to analytic number theory}. Undergraduate Texts in Mathematics. Springer-Verlag, New York, 1976.

\bibitem{BH} \textsc{Boneh, A. and Hofri, M.}: The coupon collector problem revisited --a survey of engineering problems and computational methods. \textit{J. Computat. App. Math.} {\bf 67} (1996), 277--289.

\bibitem{BP} \textsc{Boneh, S. and Papanicolaou, V.}: General asymptotic estimates for the coupon collector problem. \textit{J. Comput. App. Math.} {\bf 67} (1996), 277--289.

\bibitem{CB2001} \textsc{Cai, J.-Y. and Bach, E.}: On testing for zero polynomials by a set of points with bounded precision. In \textit{COCOON 2001}, 473--482. Lect. Notes Comput. Sc. 2108, Springer Verlag, 2001.

\bibitem{Ce1} \textsc{Ces\`{a}ro, E.}:  \'{E}tude moyenne du plus grand commun diviseur de deux nombres. \textit{Annali di Matematica Pura ed Applicata} {\bf 13} (1885), 235--250.

\bibitem{Ce3} \textsc{Ces\`{a}ro, E.}:  Sur le plus grand commun diviseur de plusieurs nombres. \textit{Ann. Mat. Pur. Appl.}~{\bf 13} (1885), 291--294.

\bibitem{Ch} \textsc{Christopher, J.}: The asymptotic density of some $k$ dimensional sets. \textit{Amer. Math. Monthly}~{\bf 63} (1956), no. 6, 399--401.

\bibitem{Co} \textsc{Cohen, E.}: Arithmetical functions of a greatest common divisor I. \textit{Proc. Amer. Math. Soc}~{\bf 11} (1960), no. 2, 164--171.

\bibitem{ED2004} \textsc{Erd\"{o}s, P. and Diaconis, P.}: On the distribution of the greatest common divisor. 
In \textit{A Festschrift for Herman Rubin}, 56--61. Lecture Notes, Monograph Series, vol. 45, Institute of Mathematical Statistics, 2004. (Reprint of the original Technical Report no. 12, Stanford University, 1977).

\bibitem{FF1} \textsc{Fern\'{a}ndez, J.\,L. and Fern\'{a}ndez, P.}: Equidistribution and coprimality. Preprint 2013.

\bibitem{FF2} \textsc{Fern\'{a}ndez, J.\,L. and Fern\'{a}ndez, P.}: Asymptotic normality of  greatest common divisors. Preprint, 2013, \texttt{arXiv:\,1302.2357}.

\bibitem{FGT} \textsc{Flajolet, P., Gardy, D. and Thimonier, L.}: Birthday paradox, coupon collectors, catching algorithms and self-organizing search. \textit{Discrete Appl. Math} {\bf 39} (1992), 207--229.

\bibitem{GS} \textsc{Gould, H.\,W. and Shonhiwa, T.}: Functions of gcds and lcms. \textit{Indian J. Math.} {\bf 39} (1997), 11--35.

\bibitem{HW} \textsc{Hardy, G.\,H. and Wright, E.\,M.}: \textit{An introduction to the Theory of Numbers}, fifth edition. Oxford Science Publications, Oxford, 1979.

\bibitem{Hwang} \textsc{Hwang, H.-K.}: {Asymptotic behaviour of some infinite products involving prime numbers}. \textit{Acta Arith.} {\bf 75} (1996), 339-350. \textit{Corrigenda} in \textit{Acta Arith.} {\bf 87} (1999), 391.

\bibitem{HS} \textsc{Herzog, F. and Stewart, B.}: Patterns of visible and non visible lattices. \textit{Amer. Math. Monthly} {\bf 78} (1971), 487--496.

\bibitem{Landau} \textsc{Landau, E.}: \textit{Neuer Beweis der Gleichung $\sum_{k=1}^\infty \mu(k)/k=0$}. [New proof of the equation $\sum_{k=1}^\infty \mu(k)/k=0$]. Inaugural dissertation, Berlin, 1899.
English translation by Michael J. Coons available at \texttt{http://arxiv.org/PS\_cache/arxiv/pdf/0803/0803.3787v2.pdf}.

\bibitem{Ny} \textsc{Nymann, J.\,E.}: On the probability that $k$ positive integers are relatively prime. \textit{J. Number Theory} {\bf 4} (1972), no. 5, 469--473.

\bibitem{To2004} \textsc{Toth, L.}: The probability that $k$ positive integers are pairwise relatively prime. \textit{Fibonacci Quart.} {\bf40} (2002), 13--18.

\end{thebibliography}
\end{document}